\documentclass[12pt,notitlepage]{amsart}
\usepackage{latexsym,amsfonts,amssymb,amsmath,amsthm}
\usepackage{graphicx}
\usepackage{color}
\usepackage[inner=1.0in,outer=1.0in,bottom=1.0in, top=1.0in]{geometry}

\newcommand{\mz}{\ensuremath{\mathbb Z}}
\newcommand{\mr}{\ensuremath{\mathbb R}}
\newcommand{\mh}{\ensuremath{\mathbb H}}
\newcommand{\mq}{\ensuremath{\mathbb Q}}

\newcommand{\shortmod}{\ensuremath{\negthickspace \negthickspace \negthickspace \pmod}}

\newcommand{\half}{\ensuremath{ \frac{1}{2}}}

\newcommand{\intR}{\int_{-\infty}^{\infty}}

\newcommand{\sumstar}{\sideset{}{^*}\sum}

\newcommand{\BesselTransform}{B}

\theoremstyle{plain}		
	\newtheorem{mytheo}{Theorem} [section]
	
	\newtheorem{myprop}[mytheo]{Proposition}
	\newtheorem{mycoro}[mytheo]{Corollary}
     \newtheorem{mylemma}[mytheo]{Lemma}
	\newtheorem{mydefi}[mytheo]{Definition}
	
	\newtheorem{myremark}[mytheo]{Remark}

\theoremstyle{remark}

\numberwithin{equation}{section}
\numberwithin{figure}{section}

\begin{document}
\title[Weyl-type hybrid subconvexity bounds for twisted $L$-functions]{Weyl-type hybrid subconvexity bounds for twisted $L$-functions and Heegner points on shrinking sets}
\author{Matthew P. Young} 
\address{Department of Mathematics \\
	  Texas A\&M University \\
	  College Station \\
	  TX 77843-3368 \\
		U.S.A.}
\email{myoung@math.tamu.edu}

 \thanks{This material is based upon work supported by the National Science Foundation under agreement No. DMS-1101261.  Any opinions, findings and conclusions or recommendations expressed in this material are those of the authors and do not necessarily reflect the views of the National Science Foundation.}

 \begin{abstract}
 Let $q$ be odd and squarefree, and let $\chi_q$ be the quadratic Dirichlet character of conductor $q$. Let $u_j$ be a Hecke-Maass cusp form on $\Gamma_0(q)$ with spectral parameter $t_j$.  By an extension of work of Conrey and Iwaniec, we show $L(u_j \times \chi_q, 1/2) \ll_{\varepsilon} (q (1 + |t_j|))^{1/3 + \varepsilon}$, uniformly in both $q$ and $t_j$.  A similar bound holds for twists of a holomorphic Hecke cusp form of large weight $k$.  
 Furthermore, we show that $|L(1/2+it, \chi_q)| \ll_{\varepsilon} ((1 + |t|) q)^{1/6 + \varepsilon}$, improving on a result of Heath-Brown.
 
 As a consequence of these new bounds, we obtain explicit estimates for the number of Heegner points of large odd discriminant in shrinking sets. 
 \end{abstract}

 \maketitle
\section{Introduction}
\subsection{Cubic moments}
Let $\chi_q$ be a real, primitive character of conductor $q$ (odd, squarefree).  Suppose that $u_j$ is a Hecke-Maass cuspidal newform of level dividing $q$, and that $f \in B_k(q)$ where $B_k(q)$ denotes the set of weight $k$ holomorphic Hecke newforms of level dividing $q$.

In a remarkable paper in the analytic theory of $L$-functions, Conrey and Iwaniec \cite{CI} showed
\begin{equation}
\label{eq:CIholo}
 \sum_{f \in B_k(q)} L(f \times \chi_q, 1/2)^3 \ll_{k, \varepsilon} q^{1+\varepsilon},
\end{equation}
for $k \geq 12$,
and
\begin{equation}
\label{eq:CIMaass}
 \sum_{t_j \leq T} L(u_j \times \chi_q, 1/2)^3 \ll_{T, \varepsilon} q^{1+\varepsilon}.
\end{equation}
Consequently, $L(\pi \times \chi_q, 1/2) \ll q^{1/3 + \varepsilon}$ for $\pi$ associated to $f$ or $u_j$, by the nonnegativity of these central values \cite{Waldspurger} \cite{KohnenZagier} \cite{KatokSarnak} \cite{Guo}.  Since the conductor of the twisted $L$-function is $q^2$, this amounts to a Weyl-type subconvexity exponent (meaning, the exponent is $1/6$ compared to the convexity bound which has exponent $1/4$) in the $q$-aspect.  Along with the Maass forms, one naturally also includes the continuous spectrum furnished by the Eisenstein series which leads to bounds for Dirichlet $L$-functions, namely $L(1/2 +it, \chi_q) \ll_t q^{1/6 + \varepsilon}$.  This subconvexity bound of Conrey and Iwaniec gave the first improvement on the Burgess bound \cite{Burgess} of $q^{3/16 + \varepsilon}$, for real characters.  There is also work of Heath-Brown \cite{H-B1} that improves on the Burgess bound but for moduli that factor in a favorable way.

The bounds \eqref{eq:CIholo} and \eqref{eq:CIMaass} depend polynomially on $k$ and $T$, respectively, but in an unspecified way (Conrey and Iwaniec state that perhaps $k^3$ is acceptable).  Motivated by some problems related to the equidistribution of Heegner points, it is desirable to obtain bounds as strong as possible in the $T$ aspect; see Section \ref{section:Heegner} below for further discussion of applications.
Along these lines, we mention a handful of results that estimate a cubic moment in the spectral/weight aspect with fixed level (often level $1$).  
Ivi{\'c} \cite{Ivic} showed
\begin{equation}
\label{eq:IvicBound}
 \sum_{T \leq t_j \leq T+1} L(u_j, 1/2)^3 \ll T^{1+\varepsilon},
\end{equation}
where here the Maass forms are of level $1$.   For reference, Weyl's law for $\Gamma_0(q) \backslash \mh$ gives 
\begin{equation}
\sum_{T \leq t_j \leq T+1} 1 \asymp qT, 
\end{equation}
uniformly in both $q, T \gg 1$ (here the notation $f(x) \asymp g(x)$ means there exist constants $c_1, c_2 > 0$ so that $c_1 f(x) \leq g(x) \leq c_2 g(x)$ for all $x$ under consideration).  Ivi{\'c}'s approach is quite different from that of \cite{CI}, and also leads to a Weyl-type subconvexity bound for level $1$ in the archimedean (spectral parameter) aspect.  The weight $k$ analog is due to Zhao Peng \cite{Pe} who showed
\begin{equation}
 \sum_{f \in B_k(1)}  L(f, 1/2)^3 \ll k^{1+\varepsilon}.
\end{equation}
Again this implies a Weyl-type subconvexity bound in the weight aspect (with fixed level $1$).
 Analogs of these cubic moments were used by Xiaoqing Li \cite{XLi} to give the first subconvexity bound for a self-dual $L$-function on $GL_3$ in the $t$-aspect (as well as certain $GL_3 \times GL_2$ Rankin-Selberg twists).  Furthermore, Blomer \cite{BlomerGL3} obtained subconvexity for twists of a self-dual $GL_3$ form by a quadratic Dirichlet character in the $q$-aspect.
  Q. Lu \cite{Lu} has modified Xiaoqing Li's \cite{XLi} method to handle variations of \eqref{eq:IvicBound} with $t_j$ in a larger window ($T \leq t_j \leq T + T^{3/8+\varepsilon}$), and for level $q$ (however, no dependency on $q$ is given).
 
Recently, Petrow \cite{Petrow} has extended the Conrey-Iwaniec bound \eqref{eq:CIholo} to all weights $k \geq 2$ which then implies corresponding bounds for Fourier coefficients of weight $\frac{k+1}{2}$ cusp forms.  Petrow's work is complementary to our results here as we focus on large $k$ (or $T$).  The main idea of Petrow's work is the development of a Motohashi-type spectral summation formula for the cubic moment, which is crucial in establishing \eqref{eq:CIholo} in the especially interesting case $k=2$.
  
Our main result is
\begin{mytheo}
\label{thm:cubicmoment}
 With notation as above, we have
 \begin{equation}
 \label{eq:cubicmomentholomorphic}
  \sum_{f \in B_k(q)} L(f \times \chi_q, 1/2)^3 \ll_{\varepsilon} (kq)^{1+\varepsilon},
 \end{equation}
 for $\chi_{q}(-1) = i^k$ (otherwise the central values all vanish), and for $T \gg 1$,
\begin{equation}
 \sumstar_{T \leq t_j \leq T+1} L(u_j \times \chi_q, 1/2)^3 \ll_{\varepsilon} (q(1+T))^{1+\varepsilon},
\end{equation}
where the star on the sum indicates the sum is restricted to even Maass forms (again, otherwise the central values vanish).
Similarly for the Eisenstein series:
\begin{equation}
\label{eq:DirichletBound}
 \int_{T}^{T+1} |L(1/2 + it, \chi_q)|^6 dt \ll_{\varepsilon} (q(1+T))^{1+\varepsilon}.
\end{equation}
The estimate \eqref{eq:cubicmomentholomorphic} holds for any even $k \geq 12$.
\end{mytheo}
The new features here compared to the previously-mentioned results is that our estimates are completely uniform in $q$ and $T$ (or $k$) together.  This leads to a Weyl-type subconvexity exponent valid in a wide range in $(q,T)$ parameter space, namely
\begin{equation}
 L(u_j \times \chi_q, 1/2) \ll (q(1+T))^{1/3 + \varepsilon},
\end{equation}
and similarly for the holomorphic forms.  For reference, the analytic conductor of $L(u_j \times \chi_q, 1/2)$ is $q^2(1+T^2)$.  The previously best-known subconvexity bound for these twisted $L$-functions with uniformity in both $q$ and $T$ (or $k$) is apparently due to Blomer and Harcos \cite{BH}.  Their bound is of Burgess quality, meaning that in the $q$-aspect the exponent is $3/8+\varepsilon$ instead of $1/3+\varepsilon$; however, their result is more general in that it allows $u_j$ (or $f$) to have arbitrary level, and the twisting character does not have to be quadratic.

One pleasant feature of our proof is that in large part it handles the holomorphic and Maass cases simultaneously, and there is no need to use the intricate asymptotic expansions of Bessel functions uniform in both the index and the argument.

The bound \eqref{eq:DirichletBound} implies $|L(1/2 + it, \chi_q)| \ll (q(1+|t|))^{1/6 + \varepsilon}$, which improves on results of Heath-Brown \cite{H-B1} \cite{H-B2} and Huxley and Watt \cite{HuxleyWatt}, for quadratic characters.

\subsection{Arithmetical applications of the cubic moments}
We will see in Section \ref{section:Heegner} below that for applications to equidistribution of Heegner points on thin sets, the cubic moment itself is more useful than the subconvexity bound that it implies.  One easy-to-state application is the following
\begin{myprop}
\label{prop:congruence}
Suppose $-D < 0$ is a sufficiently large, odd fundamental discriminant. 
\begin{enumerate}
 \item Fix $\frac49 < \eta \leq \frac12$.  There exist solutions to $b^2 \equiv - D \pmod{4a}$ for some $a$ and $b$ with $b \asymp D^{\eta}$.
 \item Fix $\frac{1}{3} < \eta \leq \frac12$. There exist solutions to $b^2 \equiv - D \pmod{4a}$ for some $a$ and $b$ with $a \asymp D^{\eta}$.
\end{enumerate}
In both cases, the number of solutions is $\gg h(-D)/D^{1/2-\eta}$, where $h(-D)$ is the class number, however the implied constant is ineffective.
\end{myprop}
These are both special cases of a more general result on the distribution of Heegner points in shrinking sets.  See Theorem \ref{thm:Heegner} below and following discussion for elaboration.  In principle, a subconvexity bound of the form $L(u_j \times \chi_{-D}, 1/2) \ll D^{1/2-\delta} (1 + |t_j|)^{B}$ (for some $\delta, B > 0$) would lead to a version of Proposition \ref{prop:congruence} with some $\eta < 1/2$.  One purpose in this paper is to strive for a small numerical value of $\eta$.

\subsection{Acknowledgment}
I would like to thank Ian Petrow and the referee for careful readings and extensive comments that improved the paper.

\section{Applications}
\label{section:Heegner}
In this section we discuss some arithmetical applications of the new hybrid subconvexity bound.  These all can be expressed as certain lattice point estimates.  

The Heegner points of (fundamental) discriminant $-D < 0$ can be identified with the collection of $SL_2(\mz)$-orbits of binary quadratic forms $ax^2 + bxy + cy^2$ of discriminant $b^2 - 4ac = -D$.  To each such quadratic form, one has the Heegner point $\tau = \frac{-b + i \sqrt{D}}{2a}$ which of course can be chosen to lie inside the usual fundamental domain $\mathcal{F}$ for $SL_2(\mz) \backslash \mh$.  Let $\Lambda_{D}$ denote the set of $SL_2(\mz)$-classes of Heegner points of discriminant $-D$.  The cardinality of $\Lambda_{D}$ is $h(-D)$, the class number.

Duke \cite{Duke} showed that $\Lambda_{D}$ becomes equidistributed in $SL_2(\mz) \backslash \mh$ as $D \rightarrow \infty$, in part by extending work of Iwaniec \cite{IwaniecHalf} bounding the Fourier coefficients of half-integral weight cusp forms.  Duke used a period formula of Maass to relate the Weyl sums over the Heegner points to these Fourier coefficients.

Following the method of Harcos and Michel \cite{HarcosMichel}, here we directly relate the Weyl sums to central $L$-values, which is provided by a formula of 
Waldspurger/Zhang \cite{Waldspurger} \cite{Zhang}.
Suppose that $u_j$ is a Hecke-Maass cusp form, orthonormalized with $\frac{dx dy}{y^2}$ (not probability measure), and define the Weyl sum
\begin{equation}
W_{D,u_j} = \sum_{\tau \in \Lambda_D } u_j(\tau).
\end{equation}
Then (e.g., see \cite[(5.1)]{LMY} for this particular formulation entirely in terms of $L$-functions)
\begin{equation}
\label{eq:periodformula}
|W_{D,u_j}|^2=\frac{\sqrt{D}L(u_j \times \chi_{-D}, \tfrac{1}{2})L(u_j,\tfrac{1}{2})}{2 L(\mathrm{sym}^2 u_j,1)}.
\end{equation}
A similar formula holds for Eisenstein series (see \cite[(22.45)]{IK} or \cite[(5.8)]{LMY}), namely if we define $W_{D, t} = \sum_{\tau \in \Lambda_D} E(\tau, 1/2 + it)$, then
\begin{equation}
\label{eq:periodformulaEisenstein}
 |W_{D, t}| = c(D) D^{1/4} \frac{|L(1/2 + it, \chi_{-D}) \zeta(1/2 + it)|}{|\zeta(1 + 2it)|},
\end{equation}
for some function $0 < c(D) \leq 10$ (in fact $c(D)$ only depends on the number of units of $\mq(\sqrt{-D})$).

Next we will set up equidistribution of Heegner points on thin sets, somewhat similarly to \cite{YoungQUE} (which treated QUE for thin sets).  In \cite{LMY} we studied  Heegner points with varying level, which is another notion of ``thin'': the number of Heegner points is independent of the level $q$ (if say $q$ is prime), yet the volume of $\Gamma_0(q) \backslash \mh$ increases with $q$.  The difference here is that we are fixing the level to be $1$, and varying the archimedean aspect.
For a given $V \geq 1$, choose a smooth and compactly-supported function $\phi: SL_2(\mz) \backslash \mh \rightarrow \mr$, satisfying $|\Delta^n \phi(x+iy)| \leq C(n) V^{2n}$, for all $n=0,1,2,\dots$.  We view the list of $C(n)$ as fixed, and $V$ may vary with the discriminant $-D$.
\begin{mytheo}
\label{thm:Heegner}
 Let $\phi$ be as above, and let $-D <0$ be an odd fundamental discriminant.  Then
 \begin{equation}
 \label{eq:HeegnerEquidistribution}
  \sum_{\tau \in \Lambda_D} \phi(\tau) = h(-D) \int_{\mathcal{F}}  \phi(z) \frac{3}{\pi} \frac{dx dy}{y^2} + O(\| \phi \|_2 D^{5/12 + \varepsilon} V^{1+\varepsilon} + (DV)^{-100}).
 \end{equation}
 The implied constant depends only on $\varepsilon > 0$ and the $C(n)$.  If one assumes the Lindel\"{o}f hypothesis for $L(u_j \times \chi_{-D}, s)$ and $L(s,\chi_{-D})$, then \eqref{eq:HeegnerEquidistribution} holds with $D^{5/12}$ replaced by $D^{1/4}$.
\end{mytheo}
The point is that we are able to explicitly give the dependence of the error term on $\phi$ (previous works on equidistribution such as \cite{Duke} typically treated $\phi$ as fixed).  The implied constant is in principle effective (however quite difficult to compute); the ineffectivity in Proposition \ref{prop:congruence} arises from ensuring the main term is larger than the error term, which in turn relies on Siegel's lower bound on the class number.

\begin{proof}
By a spectral decomposition of $\phi$, we have
\begin{equation}
 \sum_{\tau \in \Lambda_D} \phi(\tau) = h(-D) \langle \phi, \tfrac{3}{\pi} \rangle + \sum_j \langle \phi, u_j \rangle W_{D, u_j} + (\text{Eisenstein}).
\end{equation}
We bound the spectral coefficients by
\begin{equation}
(1/4 + t_j^2)^N \langle \phi, u_j \rangle = \langle \phi, \Delta^N u_j \rangle = \langle \Delta^N \phi, u_j \rangle \ll V^{2N}.
\end{equation}
Thus if $t_j \gg V (DV)^{\varepsilon}$, then $\langle \phi, u_j \rangle$ is very small.  Therefore,
\begin{equation}
\label{eq:phidifference}
\Big|\sum_{\tau \in \Lambda_D} \phi(\tau) - h(-D) \langle \phi, \tfrac{3}{\pi} \rangle \Big| \leq \sum_{t_j \ll V (DV)^{\varepsilon}} |\langle \phi, u_j \rangle W_{D, u_j}| + |(\text{Eisenstein})| + O((DV)^{-100}).
\end{equation}
By Cauchy's and Bessel's inequalities, we have that \eqref{eq:phidifference} is
\begin{equation}
\ll \| \phi \|_2 \Big(\sum_{t_j \ll V (DV)^{\varepsilon}} |W_{D,u_j}|^2 \Big)^{1/2} + |(\text{Eisenstein})| + (DV)^{-100}.
\end{equation}
Next we use \eqref{eq:periodformula} and H\"{o}lder's inequality (with exponents $3, 3, 3$), giving that \eqref{eq:phidifference} is
\begin{equation}
\ll \| \phi \|_2 D^{\frac14+\varepsilon} V^{\frac13 + \varepsilon} \Big( \sum_{t_j \ll V (DV)^{\varepsilon}} L(u_j \times \chi_{-D}, 1/2)^3 \Big)^{\frac16} \Big(\sum_{t_j \ll V (DV)^{\varepsilon}} L(u_j, 1/2)^{3} \Big)^{\frac16} + \dots,
\end{equation}
where the dots indicate the continuous spectrum and the error term.
By Theorem \ref{thm:cubicmoment}, this is
\begin{equation}
 \ll \| \phi \|_2 D^{5/12 + \varepsilon} V^{1+\varepsilon} + (DV)^{-100}.
\end{equation}
Note that the sums over $t_j$ are over level $1$ Maass forms, which when applying Theorem \ref{thm:cubicmoment} are included into all Maass forms of level dividing $D$.  This remark explains how the Lindel\"{o}f hypothesis replaces $D^{5/12}$ by $D^{1/4}$ in \eqref{eq:HeegnerEquidistribution}.

Although we have not explicitly written the contribution of the Eisenstein series, needless to say that the same bound holds for this part as for the Maass form contribution.
\end{proof}

Next we interpret this bound arithmetically.  There are a variety of choices one can make.  One particularly simple option is to choose $\phi$ to approximate the region $\frac{1}{2V} \leq |x| \leq \frac{1}{V}$, $1 \leq y \leq 2$.  One can choose $\phi$ so that $\|\phi \|_1 \asymp V^{-1}$ and $\|\phi \|_2 \asymp V^{-1/2}$.  Theorem \ref{thm:Heegner} implies that there exists a $\tau \in \Lambda_D$ inside the support of $\phi$ 
provided that $\frac{h(-D)}{V} \gg D^{5/12+\varepsilon} V^{1/2}$, which is valid for $V \ll D^{1/18-\varepsilon}$ by Siegel's ineffective lower bound on the class number.  Therefore, there exists a solution (actually, $\gg h(-D)/V$ solutions) to $b^2 \equiv -D \pmod{4a}$ with $\frac{\sqrt{D}}{a} \asymp 1$ and $|b/a| \asymp V^{-1}$.  That is, $a \asymp \sqrt{D}$ and $b \asymp \sqrt{D}/V$, so setting $V = D^{1/2-\eta}$ we see that any $4/9 < \eta \leq 1/2$ is allowable.
This gives part (1) of Proposition \ref{prop:congruence}.

Part (2) of Proposition \ref{prop:congruence} could be proved along similar lines as part (1), giving the same value $\eta > 4/9$.  Inspired by a comment of W. Duke\footnote{Talk, ``The distribution of modular closed geodesics revisited'' at the Analysis, Spectra, and Number Theory Conference, December 2014}, we can obtain the numerical improvement as follows.  For $z \in \mathcal{F}$, let $\phi(x+iy) = g(y/V)$ where $g$ is nonnegative, has support on $[1,2]$, and satisfies $g(y) =1$ for $1.25\leq y \leq 1.75$.  Then $\langle \phi, u_j \rangle = 0$ by a direct calculation using the Fourier expansion.  Similarly, the projection of $\phi$ onto Eisenstein series picks up only the constant term of the Fourier expansion, giving
\begin{equation}
 \langle \phi, E(\cdot, 1/2 + it) \rangle = \int_0^{\infty} \Big(y^{1/2 + it} + \frac{\zeta^*(1-2it)}{\zeta^*(1+2it)} y^{1/2-it} \Big) g(y/V) \frac{dy}{y^2},
\end{equation}
which is $\ll V^{-1/2} |\widetilde{g}(-1/2 + it)|$.  Therefore, \eqref{eq:phidifference} simplifies in this case to give
\begin{equation}
 \Big|\sum_{\tau \in \Lambda_D} \phi(\tau) - h(-D) \langle \phi, \tfrac{3}{\pi} \rangle \Big| \ll \frac{D^{1/4}}{V^{1/2}} \intR |\widetilde{g}(-1/2 + it)| 
 \frac{|L(1/2 + it, \chi_{-D}) \zeta(1/2 + it)|}{|\zeta(1 + 2it)|} dt.
\end{equation}
Since $g$ is fixed, $\widetilde{g}(-1/2 + it)$ has rapid decay, and so by the original bound of Conrey-Iwaniec we have
\begin{equation}
 \sum_{\tau \in \Lambda_D} \phi(\tau) \gg \frac{h(-D)}{V} + O(V^{-\frac12} D^{\frac14 + \frac16 + \varepsilon}).
\end{equation}
The main term is larger than the error term provided $V \ll D^{1/6-\varepsilon}$.  A Heegner point in the support of this $\phi$ has
$a \asymp V^{-1} D^{1/2}$ which can be made as small as $D^{\frac12 - \frac16 + \varepsilon}$, as claimed.
These Heegner points are approaching the cusp $\infty$ fairly quickly.  The Lindel\"{o}f Hypothesis would imply that $a \gg D^{\varepsilon}$ is allowable.

Finally, we mention one other variation where we only ask for an upper bound on the number of Heegner points in a small box. 
\begin{mycoro}
\label{coro:HeegnerUpperBound}
 Fix $x_0 + i y_0 \in \mathcal{F}$, and let $V \gg 1$.  Then the number of Heegner points $\tau \in \mathcal{F}$ of discriminant $-D$ such that $|\tau - (x_0 + i y_0)| \ll V^{-1}$ is
 \begin{equation}
 \ll \frac{h(-D)}{V^2} + D^{5/12 + \varepsilon} V^{\varepsilon}.
\end{equation}
The implied constant depends on $x_0+ iy_0$ and $\varepsilon$.
\end{mycoro}
One pleasing feature of this upper bound is that it is $o(D^{1/2})$ for a wide range of values of $V$.  To prove Corollary \ref{coro:HeegnerUpperBound}, we apply Theorem \ref{thm:Heegner} with a nonnegative function $\phi$ that equals $1$ on $|\tau - (x_0 + iy_0)| \ll V^{-1}$, such that $\|\phi\|_1 \asymp V^{-2}$, and $\|\phi \|_2 \asymp V^{-1}$.

\section{High-level sketch of the method}
Here we give a brief overview of the proof focusing on the essential details.  By an approximate functional equation, it suffices to show
\begin{equation}
\label{eq:initialsum}
 \sum_{n_1, n_2, n_3 \ll (qT)^{1+\varepsilon}} \thinspace \thinspace \sumstar_{T \leq t_j \leq T+\Delta} w_j^* \frac{\chi_{q}(n_1 n_2 n_3) \lambda_j(n_1) \lambda_j(n_2) \lambda_j(n_3)}{\sqrt{n_1 n_2 n_3}} \ll (qT)^{1+\varepsilon},
\end{equation}
where $w_j^*$ are weights arising in the Kuznetsov formula, so that $\sum_{T \leq t_j \leq T+1} w_j^* \asymp qT$ (so the weights are $\asymp 1$ on average, by Weyl's law), and where $\Delta$ is an arbitrarily small power of $T$.
The Kuznetsov formula converts the sum over $t_j$ into a sum of Kloosterman sums, transforming the left hand side of \eqref{eq:initialsum} to the form 
\begin{equation}
\label{eq:HSum}
q \sum_{n_1, n_2, n_3 \ll (qT)^{1+\varepsilon}} \sum_{c \equiv 0 \shortmod{q}} \frac{1}{c} \frac{\chi_{q}(n_1 n_2 n_3) S(n_1 n_2, n_3;c)}{\sqrt{n_1 n_2 n_3}} \BesselTransform\Big(\frac{4 \pi \sqrt{n_1 n_2 n_3}}{c} \Big),
\end{equation}
where $B$ is given as a certain integral transform.  This weight function takes the rough shape
\begin{equation}
 \BesselTransform(x) \approx \frac{\Delta T}{\sqrt{x}} \cos\Big(x - 2\frac{T^2}{x} + \dots\Big),
\end{equation}
and is very small for $x \ll \Delta T^{1-\varepsilon}$, meaning that we may truncate the sum over $c$ at $\sqrt{n_1 n_2 n_3}/(\Delta T)$.  Actually, we need to treat two different types of $\BesselTransform$ because we need to restrict to a sum over the even part of the spectrum, but for Theorem \ref{thm:cubicmoment} both cases are fairly similar.
Suppose that each $n_i \asymp N_i$, where $N_i \ll (qT)^{1+\varepsilon}$.

Following \cite{CI}, we apply Poisson summation to each sum over $n_i$ modulo $c$, transforming \eqref{eq:HSum} into an expression of the form
\begin{equation}
  \sum_{\substack{m_1, m_2, m_3}} \sum_{c \equiv 0 \shortmod{q}} \frac{q}{c} G(m_1, m_2, m_3;c) K(m_1, m_2, m_3, c),
\end{equation}
where with $e_c(x) = e(x/c)$, we define
 \begin{equation}
 \label{eq:Gdef}
G(m_1 ,m_2, m_3;c) = c^{-3} \sum_{a_1, a_2, a_3 \shortmod{c}} \chi_q(a_1 a_2 a_3) S(a_1 a_2, a_3;c) e_c(a_1m_1 + a_2 m_2 + a_3 m_3),
\end{equation}
and
\begin{equation}
K(m_1 ,m_2, m_3, c) = (N_1 N_2 N_3)^{-1/2} \int_{\mr^3} \BesselTransform\Big(\frac{4 \pi \sqrt{t_1 t_2 t_3}}{c} \Big) e_c(-m_1 t_1 - m_2 t_2 - m_3 t_3) dt_1 dt_2 dt_3.
\end{equation}
Conrey and Iwaniec evaluated $G$, giving that if some mild coprimality restrictions are in place then with $c = qr$, we have
\begin{equation}
 G(m_1 ,m_2, m_3;c) = \frac{\chi_{q}(-1) e_c(m_1 m_2 m_3)}{q^2 r} H(\overline{r} m_1 m_2 m_3;q),
\end{equation}
where $H$ is a certain two-variable complete character sum (see \eqref{eq:Hwqdef} below for its definition or \eqref{eq:H*intro} for a close variant).  
As for $K$, by an elaborate stationary phase analysis, we obtain that
\begin{equation}
\label{eq:KapproxSketchSection}
K(m_1 ,m_2, m_3, c) \approx \frac{c^2 \Delta T}{(N_1 N_2 N_3)^{1/2}} e_c(-m_1 m_2 m_3)  e\Big(\alpha \frac{T^2 c}{m_1 m_2 m_3} + \dots\Big),
\end{equation}
where $\alpha \neq 0$ is a fixed constant.  Furthermore, $K$ is very small unless $m_i \asymp M_i = \sqrt{N_1 N_2 N_3}/N_i$; this is already a square-root savings in each variable due to a reduction in length compared to $N_i$.  Note the quite remarkable cancellation in the primary phase $e_c(m_1 m_2 m_3)$, as well as the common feature that $G$ and $K$ essentially only depend on $m_1 m_2 m_3/r$ and $q$.  This feature allows for a particularly efficient separation of variables.  
Our work departs from \cite{CI} in the analysis of $K$.  When $T$ is large, then $K$ multiplied by $e_c(m_1 m_2 m_3)$ is oscillatory which makes the separation of variables nontrivial.

Next we discuss how the variables are separated, first arithmetically and then analytically.  It suffices to consider a variant of $H$ where the variables inside the summation have a coprimality condition, that is,
\begin{equation}
\label{eq:H*intro}
 H^*(w;q) = \sum_{\substack{u, v \shortmod{q} \\ (uv-1, q) =1}} \chi_{q}(uv(u+1)(v+1)) e_q((uv-1)w),
\end{equation}
in which case from \cite[(11.9)]{CI}, 
\begin{equation}
 H^*(w;q) = \frac{1}{\phi(q)} \sum_{\psi \shortmod{q}} \tau(\overline{\psi}) g(\chi, \psi) \psi(w).
\end{equation}
Here $g(\chi, \psi)$ is $O(q^{1+\varepsilon})$ by Deligne's bound, but otherwise we do not need any properties of $g$.  A similar separation of variables applies to $K$ by the Mellin transform, giving
\begin{equation} K(m_1, m_2, m_3, c) \approx \frac{c^{3/2} \Delta}{(N_1 N_2 N_3)^{1/4}} e_c(-m_1 m_2 m_3)  L(m_1, m_2, m_3, c),
\end{equation}
where
\begin{equation}
 L(m_1, m_2, m_3, c) = \int_{|u| \ll U} \ell(u) \Big(\frac{m_1 m_2 m_3}{c}\Big)^{iu} du, \qquad 
%
%
 U = \frac{T^2 c}{(N_1 N_2 N_3)^{1/2}} \ll \frac{T}{\Delta},
\end{equation}
and $\ell(u) \ll 1$ slightly depends on the variables $m_1, m_2, m_3, c$, but for this sketch we pretend that it does not.  

Therefore, in all, we have transformed \eqref{eq:initialsum} into an expression of the form
\begin{equation}
 \sum_{\substack{m_1, m_2, m_3 \\ m_i \asymp M_i}} \sum_{r} \frac{C^{3/2} \Delta}{(N_1 N_2 N_3)^{1/4}}  \frac{1}{\phi(q)} \sum_{\psi \shortmod{q}} \frac{\tau(\overline{\psi}) g(\chi, \psi) }{ (q r)^2} \psi(\overline{r} m_1 m_2 m_3)   \int_{|u| \ll U} \ell(u) \Big(\frac{m_1 m_2 m_3}{qr}\Big)^{iu} du,
\end{equation}
where we have further restricted to $c = qr \asymp C$ (with $C \ll \sqrt{N_1 N_2 N_3}/(\Delta T)$).  
At this point the expression can be pleasantly arranged into a bilinear structure as
\begin{equation}
 \sum_{\psi \shortmod{q}} \int_{|u| \ll U} \frac{\Delta q^{-\frac{1}{2}+\varepsilon} |\ell(u) g(\chi, \psi) |}{C^{\frac12} (N_1 N_2 N_3)^{\frac14}} \Big| \sum_{m_1, m_2} \psi( m_1m_2) (m_1 m_2)^{iu}  \Big| \cdot  \Big| \sum_{m_3, r} \psi(\overline{r} m_3) \Big(\frac{m_3}{r}\Big)^{iu}  \Big|   du.
\end{equation}
The hybrid large sieve inequality of Gallagher \cite{Gallagher} bounds this by
\begin{equation}
 \frac{\Delta q^{1/2+\varepsilon} }{C^{1/2} (N_1 N_2 N_3)^{1/4}} \Big(q U + M_1 M_2 \Big)^{1/2} \Big(q U + \frac{M_3 C}{q} \Big)^{1/2}  \Big(\frac{M_1 M_2 M_3 C}{q}\Big)^{1/2}.
\end{equation}
An easy calculation shows this is bounded by
\begin{equation}
 \Delta   \Big(\frac{qT}{\Delta} \Big)^{1/2} (qT)^{1/2+\varepsilon} =  \Delta^{1/2} (qT)^{1 + \varepsilon},
\end{equation}
consistent with our claim \eqref{eq:initialsum} (taking $\Delta = T^{\varepsilon}$).

\section{Initial setup}
Many of the early structural steps are similar to \cite{CI} (and could now be considered standard), so in places we will refer to \cite{CI} for the details.

Let $T \geq 100$ and $q$ be an odd squarefree integer, and let $2 \leq \Delta < T/2$.  Towards \eqref{eq:cubicmomentholomorphic}, we shall obtain the bound
\begin{equation}
\label{eq:cubicmomentholmorphicSumOverWeight}
 \sumstar_{T \leq k \leq T + \Delta} \thinspace \sum_{f \in B_k(q)} w_f^* L(f \times \chi_q, 1/2)^3 \ll \Delta T q (Tq)^{\varepsilon},
\end{equation}
where the star indicates that $k$ (necessarily even) satisfies $\chi_q(-1) = i^k$, which in turn simply fixes $k \pmod{4}$ (assuming $q$ is chosen).  For the Maass forms, we will show
\begin{equation}
\label{eq:cubicmomentMaass}
 \sumstar_{T \leq t_j \leq T+\Delta} w_j^* L(u_j \times \chi_q, 1/2)^3 \ll \Delta T q (Tq)^{\varepsilon},
\end{equation}
where $w_f^*$ and $w_j^*$ are certain weights arising in the Petersson/Kuznetsov formula, satisfying $w_f^* \gg (kq)^{-\varepsilon}$, $w_j^* \gg (Tq)^{-\varepsilon}$.
These are the same weights used by Conrey and Iwaniec, up to a simple scaling, so we refer to \cite{CI} for details on these weights.  The sums are over newforms of level dividing $q$.  A similar bound holds for the Eisenstein series, namely
\begin{equation}
\label {eq:cubicmomentEisenstein}
 \int_{T}^{T + \Delta} |L(1/2 + it, \chi_q)|^6 dt \ll \Delta T q (Tq)^{\varepsilon}.
\end{equation}
All the twisted $L$-functions are newforms of level $q^2$.

In our work, we will assume $T \gg q^{\eta}$ for some fixed $\eta > 0$; in practice this will mean that expressions of the form $O(T^{-A})$ with $A > 0$ arbitrarily large, are $O((q T)^{-A'})$, with $A'$ arbitrarily large.  We are able to restrict to this case because Conrey and Iwaniec showed an upper bound of the form $T^A q^{1+\varepsilon}$ (with some fixed but unspecified $A > 0$) for the cubic moment, so if $T \ll q^{\eta}$ with $\eta$ arbitrarily small, then their bound is satisfactory for Theorem \ref{thm:cubicmoment}.  We will also suppose that $T^{\eta'} \asymp \Delta $ for some $\eta' > 0$ fixed but arbitrarily small; note that if we prove \eqref{eq:cubicmomentholmorphicSumOverWeight}--\eqref{eq:cubicmomentEisenstein} for such $\Delta$, then it extends to larger values of $\Delta$ automatically by dissecting the longer interval into these shorter pieces.

%
%

In the case of Maass forms, it is technically convenient to introduce the function
\begin{equation}
 h(t) = 
 \frac{1}{\cosh\Big(\frac{t-T}{\Delta} \Big)} + \frac{1}{\cosh\Big(\frac{t+T}{\Delta} \Big)}.
\end{equation}
This choice of $h$ is analytic for $|\text{Im}(t)| < \frac{\pi}{2} \Delta$ (which we may assume is large), nonnegative on $\mr \cup \{ iy : -1 \leq y \leq 1 \}$, and even.  Furthermore, $h(t) \gg 1$ for $T \leq t \leq T+\Delta$, and $h(t)$ is very small outside of $|t \mp  T| \leq \Delta$.  To prove \eqref{eq:cubicmomentMaass}, it therefore suffices to show
\begin{equation}
 \sum_{t_j} w_j^* h(t_j) L(u_j \times \chi_q, 1/2)^3 \ll \Delta Tq (Tq)^{\varepsilon}.
\end{equation}
For the holomorphic case, it is also convenient to sum over $k$ with a smooth weight function but in this case we can take $h(k) = w(\frac{k-1-2T}{\Delta})$ where $w$ is any smooth, nonnegative, compactly-supported function   where $w(t) = 1$ for $1 \leq t \leq 2$, and zero for $t \leq 1/2$ and $t \geq 3$ (here we chose to center $w$ at $2T + 1$ to simplify some later formulas).  


After some initial steps, eventually our method handles the Maass and holomorphic cases in a unified way.  We shall more carefully treat the Maass case since it is a bit more complicated.

\section{Approximate functional equation and separation of variables}
Since all the $L$-functions under consideration have functional equation $+1$, we have by \cite[Theorem 5.3]{IK}
\begin{equation}
 L(u_j \times \chi_q, 1/2) = 2 \sum_{n=1}^{\infty} \frac{\lambda_j(n) \chi_q(n)}{\sqrt{n}} V_{t_j}(n/q),
\end{equation}
where
\begin{equation}
 V_r(y) = \frac{1}{2 \pi i} \int_{(1)} \frac{\Gamma\Big(\frac{1/2+s+ir}{2} \Big) \Gamma\Big(\frac{1/2 + s-ir}{2} \Big)}{\Gamma\Big(\frac{1/2+ir}{2} \Big) \Gamma\Big(\frac{1/2-ir}{2} \Big)} G_{}(s) (\pi y)^{-s} \frac{ds}{s}.
\end{equation}
Here $G_{}(s)$ is an even, analytic function satisfying $G_{}(0) = 1$.  We choose $G_{}(s) = e^{s^2}$ (this $G$ should not be confused with \eqref{eq:Gdef}).
In order to more simply sum over $t_j$, we wish to separate the variables $r$ and $y$ in $V_r(y)$.  By symmetry, we may as well suppose $r > 0$.

By Stirling, if $|\text{Im}(z)| \rightarrow \infty$ (with fixed real part), but $|s| \ll |z|^{1/2}$, then
\begin{equation}
 \frac{\Gamma(z+s)}{\Gamma(z)} = z^{s} \Big(1 + \sum_{k=1}^{N} \frac{P_k(s)}{z^k} + O\Big(\frac{(1+|s|)^{2N+2}}{|z|^{N+1}} \Big)\Big),
\end{equation}
for certain polynomials $P_k(s)$ of degree $\leq 2k$.    Since $G(s)$ has exponential decay, we may truncate at $|\text{Im}(s)| \ll T^{\varepsilon}$ with only a small error in $V_r(y)$.
With these assumptions on $s$ and $r$, we have that 
\begin{equation}
 \frac{\Gamma\Big(\frac{1/2+s+ir}{2} \Big) \Gamma\Big(\frac{1/2 + s-ir}{2} \Big)}{\Gamma\Big(\frac{1/2+ir}{2} \Big) \Gamma\Big(\frac{1/2-ir}{2} \Big)} = (r/2)^s \Big(1 + \sum_{k=1}^{N} \frac{P_k(s)}{r^k} + O\Big(\frac{(1+|s|)^{2N+2}}{r^{N+1}}\Big)\Big),
\end{equation}
for a different collection of polynomials $P_k(s)$.  For convenience, set $P_0(s) = 1$.

Hence
\begin{equation}
\label{eq:Vrasymptotic}
 V_r(y) = \sum_{k=0}^{N} r^{-k} \frac{1}{2 \pi i} \int_{(1)} G(s) P_k(s) \Big(\frac{r}{2 \pi y}\Big)^s \frac{ds}{s} + O\Big(r^{-N-1} \Big(1+ \frac{y}{r}\Big)^{-A}\Big),
\end{equation}
where the extra factor $(1+ \frac{y}{r})^{-A}$ arises from moving the contour to $\text{Re}(s) = A$ if $y \geq r$, and to $\text{Re}(s) = -1$ if $y \leq r$.  We further refine \eqref{eq:Vrasymptotic} by approximating $r$ by $T$, which is a good approximation since in our application $h(r)$ is very small unless $|r-T| \ll \Delta \log{T}$, and $\Delta$ is a small power of $T$.  We may assume that $|s \frac{(r-T)}{T}| \ll 1$, so that
by expansion into Taylor series we have
\begin{equation}
 r^{s} = T^{s} e^{s \log(1 + \frac{r-T}{T})} = T^s \sum_{l=0}^{N} Q_l(s) \Big(\frac{r-T}{T} \Big)^l  + O\Big( (1 + |s|)^{N+1}  \Big(\frac{|r-T|}{T} \Big)^{N+1} \Big),
\end{equation}
for certain polynomials $Q_l(s)$ of degree $\leq l$.
We then derive for $r \asymp T$ that
\begin{equation}
 V_r(y) = \sum_{k=0}^{N} \sum_{l=0}^{N} \frac{1}{T^k} \Big(\frac{r-T}{T} \Big)^l  V_{k,l}\Big(\frac{y}{T} \Big) + O\Big(\Big(\frac{1 +|r-T|}{T} \Big)^{N+1} \Big(1 + \frac{y}{T}\Big)^{-A} \Big),
\end{equation}
where $V_{k,l}$ is a function of the form
\begin{equation}
\label{eq:Vkl}
 V_{k,l}(y) = \frac{1}{2 \pi i} \int_{(1)} P_{k,l}(s) (2 \pi y)^{-s} G(s) \frac{ds}{s},
\end{equation}
for some polynomial $P_{k,l}$ of degree $\leq 2k + l$.

In light of this form of the approximate functional equation, it suffices to show
\begin{equation}
\label{eq:spectralsum2}
\sumstar_{t_j} h_{k,l}(t_j) w_j^* \sum_{n_1, n_2, n_3=1}^{\infty} \frac{\chi_q(n_1 n_2 n_3) \lambda_j(n_1) \lambda_j(n_2) \lambda_j(n_3)}{(n_1 n_2 n_3)^{1/2}} \prod_{i=1}^{3} V_i\Big(\frac{n_i}{Tq}\Big) + (\text{cts}) \ll \Delta (qT)^{1+\varepsilon},
\end{equation}
where $(\text{cts})$ represents the obvious continuous spectrum contribution (which is also nonnegative), and
\begin{equation}
\label{eq:hkldef}
h_{k,l}(r) = \Big((r-T)^l T^{-k-l} + (-r-T)^l T^{-k-l} \Big) h(r), 
\end{equation}
and $V_i$ are functions of the form \eqref{eq:Vkl}.  To prepare this for the Kuznetsov formula, we use the Hecke relation $\lambda_j(n_1) \lambda_j(n_2) = \sum_{d | (n_1, n_2)} \lambda_j(\frac{n_1 n_2}{d^2})$, giving that \eqref{eq:spectralsum2} equals
\begin{equation}
\label{eq:spectralsum}
\sum_{(d,q)=1} d^{-1} \sum_{n_1, n_2, n_3} V_1\Big(\frac{n_1}{d^{-1} Tq}\Big)V_2\Big(\frac{n_2}{d^{-1} Tq}\Big) V_3\Big(\frac{n_3}{Tq}\Big) \frac{\chi_{q}(n_1 n_2 n_3)}{(n_1 n_2 n_3)^{1/2}}  \sumstar_{j} h_{k,l}(t_j) w_j^* \lambda_j(n_1 n_2) \lambda_j(n_3).
\end{equation}
The Kuznetsov formula says
\begin{multline}
\label{eq:avgKuznetsov}
\sumstar_j h_{k,l}(t_j) w_j^* \lambda_j(n_1 n_2) \lambda_j(n_3) + (\text{Continuous})
\\
= D\delta_{n_1 n_2 = n_3}  + \sum_{\pm} q \sum_{c \equiv 0 \shortmod{q}} \frac{S(n_1 n_2, \pm n_3;c)}{c} \BesselTransform^{\pm} \Big(\frac{4 \pi \sqrt{ n_1 n_2 n_3}}{c} \Big),
\end{multline}
where $D$ is the size of the diagonal term.  We have chosen our weights $w_j^*$ so that $D \asymp \Delta T q$, and it is easy to see that the bound arising from the diagonal terms is of the desired magnitude.  Here
\begin{equation}
\label{eq:H+def}
\BesselTransform^+(x) = \frac{i}{\pi} \intR \Big(\frac{J_{2ir}(x) - J_{-2ir}(x)}{\cosh( \pi r)} \Big) h_{k,l}(r) r dr,
\end{equation}
and
\begin{equation}
 \label{eq:H-def}
\BesselTransform^-(x) = \frac{4}{\pi^2} \intR K_{2ir}(x) \sinh(\pi r) h_{k,l}(r) r dr.
\end{equation}

To prove Theorem \ref{thm:cubicmoment} (that is, the Maass and Eisenstein cases of the theorem), it suffices to show the following
\begin{mytheo}
\label{thm:quadlinearbound}
Let $1 \ll N_1, N_2, N_3 \ll (qT)^{1+\varepsilon}$, and let each $w_i$ be a smooth weight function with support on $x \asymp N_i$, and satisfying $w_i^{(k)}(x) \ll N_i^{-k}$.  Suppose that $\BesselTransform$ is given by \eqref{eq:H+def} or \eqref{eq:H-def}, with $h = h_{k,l}$ of the form \eqref{eq:hkldef}. 
Then
\begin{multline}
\label{eq:quadsum}
 \sum_{ n_1, n_2, n_3}  w_1(n_1) w_2(n_2) w_3(n_3) \chi_{q}(n_1 n_2 n_3) \sum_{c \equiv 0 \shortmod{q}} \frac{S(n_1 n_2, \pm n_3;c)}{c} \BesselTransform\Big(\frac{4 \pi \sqrt{n_1 n_2 n_3}}{c} \Big) 
\\
\ll (N_1 N_2 N_3)^{1/2}  \Delta T (qT)^{\varepsilon}.
\end{multline}
\end{mytheo}
This follows by applying a smooth dyadic partition of unity to each $n_i$-sum.

Next we reduce the holomorphic case of Theorem \ref{thm:cubicmoment} (that is, \eqref{eq:cubicmomentholomorphic})   to a variation on Theorem \ref{thm:quadlinearbound}.  The separation of variables in the approximate functional equation is quite similar to the Maass case, so we omit the details.  In this way, we quickly reduce to estimating an expression of the form \eqref{eq:spectralsum} but with the sum over $t_j$ replaced with a sum of the form
\begin{equation}
\label{eq:holomorphicsum}
 \sum_{k \equiv a \shortmod{4}} w\Big(\frac{k-1-2T}{\Delta} \Big) \sum_{f \in B_k(q)} w_f^* \lambda_f(n_1 n_2) \lambda_f(n_3).
\end{equation}
By the Petersson formula, \eqref{eq:holomorphicsum} equals
\begin{equation}
 D \delta_{n_1 n_2 = n_3} + q \sum_{c \equiv 0 \shortmod{q}} \frac{S(n_1 n_2, n_3;c)}{c} \BesselTransform^{\text{holo}}\Big(\frac{4 \pi \sqrt{n_1 n_2 n_3}}{c}\Big),
\end{equation}
where $D \asymp \Delta T q$, and
\begin{equation}
\label{eq:Hholodef}
 \BesselTransform^{\text{holo}}(x) = T \sum_{k \equiv a \shortmod{4}}  w\Big(\frac{k-1-2T}{\Delta} \Big) J_{k-1}(x).
\end{equation}
Here the factor $T$ arises because we chose the weights $w_f^*$ so that $\sum_{f \in B_k(q)} w_f^* \asymp k \asymp T$ (often in the literature, e.g. \cite[Theorem 3.6]{IwaniecClassical} the weights do not grow with $k$ so here it is necessary to multiply by $T$).
At this point we have reduced the proof of \eqref{eq:cubicmomentholomorphic} to extending Theorem \ref{thm:quadlinearbound} to hold for $\BesselTransform = \BesselTransform^{\text{holo}}$.

\section{Exponential integrals}
In Section \ref{section:separationAnalytic} below we require extensive information on some oscillatory integrals.  The stationary phase estimates we need are in principle standard, but the error terms occuring in the literature are generally not good enough for our purposes here.  The issue is that Theorem \ref{thm:cubicmoment} needs to hold uniformly in $(q,T)$-parameter space.  The weight function analysis is entirely in the $T$-aspect, and it needs to commute with the $q$-aspect analysis.  In practice this means that we require the error terms in the $T$-aspect to be very strong, as otherwise if $T$ is only a small power of $q$ then this error term is not much smaller than the main term, and we cannot apply savings in the $q$-aspect on the error term from the $T$-aspect.  For this reason, we will quote some convenient results of \cite{BKY} that have sufficiently strong errors.

First we begin with a useful definition.
\begin{mydefi}
\label{def:inert}
Suppose that $f(x_1, \dots, x_n)$ is a smooth function on $\mr^n$.  We say that $f$ is \emph{inert} if 
\begin{equation}
 x_1^{i_1} \dots x_n^{i_n} f^{(i_1, \dots, i_n)}(x_1, \dots, x_n) \ll 1,
\end{equation}
with an implied constant depending on $i_1, \dots, i_n$ and with the superscript denoting partial differentiation.  
\end{mydefi}
\begin{myremark}
 In practice we require that the implied list of constants is \emph{uniform} in terms of certain parameters (e.g., $q,T, \Delta, N_i, C$, but not necessarily $\varepsilon$).  It is then appropriate to say that $f$ is \emph{uniformly} inert (in terms of those parameters).
\end{myremark}

We remark that an inert function that is also say Schwartz class (e.g., with compact support) can have its variables separated almost for free, in the sense that
\begin{equation}
\label{eq:fourierseparation}
 f(x_1, \dots, x_n) =  \int_{\mr^n} \widehat{f}(y_1, \dots, y_n) e(x_1 y_1) \dots e(x_n y_n) dy_1 \dots dy_n,
\end{equation}
where $\widehat{f}(y_1, \dots, y_n) \ll (1 + |y_1|)^{-A} \dots (1+ |y_n|)^{-A}$.  There exists a similar Mellin formula, of course.
Note that the product of two inert functions is also inert (with new implied constants of course).  Also if $f(t)$ is inert and $\alpha_1, \dots, \alpha_n \in \mr$, then 
\begin{equation}
\label{eq:inertprecomposition}
g(x_1, \dots, x_n) = f(\prod_{i=1}^{n} (x_i/X_i)^{\alpha_i}) 
\end{equation}
is inert (with uniformity in the $X_i$ but not the $\alpha_i$).  Virtually all our constructions of inert functions are variations of \eqref{eq:inertprecomposition}.


Next we synthesize both Lemma 8.1 and Proposition 8.2 of \cite{BKY} along with some simplified choices of parameters, with the following
\begin{mylemma}
\label{lemma:exponentialintegral}
 Suppose that $w$ is a smooth weight function with compact support on $[X, 2X]$,
satisfying $w^{j}(t) \ll X^{-j}$, for $X \gg 1$ (in particular, $w$ is inert with uniformity in $X$).  Also suppose that $\phi$ is smooth and satisfies $\phi^{(j)}(t) \ll \frac{Y}{X^j}$ for some $Y \gg X^{\varepsilon}$.  Let
\begin{equation}
 I = \intR w(t) e^{i \phi(t)} dt.
\end{equation}
\begin{enumerate}
 \item If $\phi'(t) \gg \frac{Y}{X}$ for all $t$ in the support of $w$, then $I \ll_A Y^{-A}$ for $A$ arbitrarily large.
 \item If $\phi''(t) \gg \frac{Y}{X^2}$ for all $t$ in the support of $w$, and there exists $t_0 \in \mr$ such that $\phi'(t_0) = 0$ (note $t_0$ is necessarily unique), then
 \begin{equation}
  I = \frac{e^{i \phi(t_0)}}{\sqrt{\phi''(t_0)}} F(t_0) + O(Y^{-A}),
 \end{equation}
where $F$ is an inert function (depending on $A$, but uniformly in $X$ and $Y$) supported on $t_0 \asymp X$. 
\end{enumerate}
\end{mylemma}

As a fairly direct consequence of Lemma \ref{lemma:exponentialintegral}, we shall obtain the following asymptotic formula for a certain $2$-dimensional oscillatory integral.  This will be used in Section \ref{section:separationAnalytic}.
\begin{mylemma}
\label{lemma:Iuasymptotic}
Suppose $\alpha, \beta, \gamma \in \mr$, let $X \gg 1$, and
suppose that $f_1(t_1)$ and $f_2(t_2)$ are (uniformly) inert functions with support on $t_1 \asymp X_1$, $t_2 \asymp X_2$ with $X_1, X_2 \gg X$.  If the following inequalities are true: $\frac{\alpha}{X_1 X_2} \gg X$, $\beta X_1 \gg X$, $\gamma X_2 \gg X$, then
 \begin{multline}
 \label{eq:doubleintegral}
  \intR \intR f_1(t_1) f_2(t_2) e\Big(-\frac{\alpha}{t_1 t_2} - \beta t_1 - \gamma t_2 \Big) dt_1 dt_2 = \Big(\frac{X_1 X_2}{\beta \gamma} \Big)^{1/2} e(-3 (\alpha \beta \gamma)^{1/3}) f_3(\alpha, \beta, \gamma)
  \\
  + O_A((X_1 + X_2)X^{-A}),
 \end{multline}
 where $f_3$ is an inert function (depending on $A$, uniform in $X_1, X_2, X$), having support on
 \begin{equation}
 \label{eq:alphabetagammasupport}
  \frac{\alpha}{\beta} \asymp X_1^2 X_2, \qquad \frac{\alpha}{\gamma} \asymp X_1 X_2^2. 
 \end{equation}
More generally, if we assume that 
 \begin{equation}
 \label{eq:parametersizes}
 \frac{|\alpha|}{X_1 X_2} \gg X, \qquad |\beta| X_1 \gg X, \qquad |\gamma| X_2 \gg X,
 \end{equation}
 then the integral in \eqref{eq:doubleintegral} is $O((X_1 + X_2)X^{-A})$ if  $\alpha, \beta, \gamma$ do not all have the same sign.  If $\alpha, \beta, \gamma < 0$, then \eqref{eq:doubleintegral} remains valid, with the convention $(-1)^{1/3} = -1$.  Finally, if exactly one or exactly two of the inequalities \eqref{eq:parametersizes} are valid, then the integral in \eqref{eq:doubleintegral} is $O((X_1 + X_2)X^{-A})$.
\end{mylemma}
\begin{proof}[Proof of Lemma \ref{lemma:Iuasymptotic}]
The basic idea is to use stationary phase analysis in each variable.  Let us first examine the $t_1$ integral, under the assumption $\alpha, \beta, \gamma > 0$.  Let $Z_1 = \frac{\alpha}{X_1 X_2} + \beta X_1$.
By repeated integration by parts (Lemma \ref{lemma:exponentialintegral}, part 1), the integral is small (namely, $O_A(Z_1^{-A}) = O_A(X^{-A})$) except possibly if $\frac{\alpha}{X_1 X_2} \asymp \beta X_1$.  Furthermore, this argument shows the integral is very small unless $\alpha$ and $\beta$ have the same sign (so by symmetry, $\alpha$ and $\gamma$ have the same sign too).
There exists a stationary point at $t_1^0 = \sqrt{\frac{\alpha}{\beta t_2}}$ .  Therefore, Lemma \ref{lemma:exponentialintegral} gives
\begin{equation}
 \intR f_1(t_1) e\Big(-\frac{\alpha}{t_1 t_2} - \beta t_1 \Big) dt_1 = \frac{X_1^{1/2}}{\beta^{1/2}} e(-2\frac{\sqrt{\alpha \beta}}{\sqrt{t_2}}\Big) F(t_2) + O(Z_1^{-A}),
\end{equation}
where $F(t_2)$ is a function that is inert in terms of $\alpha, \beta, t_2$ (see \eqref{eq:inertprecomposition} and surrounding discussion), and has support on an interval of the form $t_2 \asymp \frac{\alpha}{\beta X_1^2}$.

Next we insert this expansion into the $t_2$ integral, so we need to evaluate
\begin{equation}
\label{eq:t2integral}
 \intR  e(-2\frac{\sqrt{\alpha \beta}}{\sqrt{t_2}} - \gamma t_2 \Big)  f_2(t_2) F(t_2) dt_2.
\end{equation}
Again, \eqref{eq:t2integral} is $O(Z_2^{-A})$, where $Z_2 = \frac{\alpha}{X_1 X_2} + \gamma X_2$, except possibly if $X_2 \asymp t_2^0 := \frac{(\alpha \beta)^{1/3}}{\gamma^{2/3}}$, in which case there is a stationary point at $t_2^0$.  So we obtain that \eqref{eq:t2integral} equals
\begin{equation}
 \frac{X_2^{1/2}}{\gamma^{1/2}} e(-3(\alpha \beta \gamma)^{1/3}) f(\alpha, \beta, \gamma) + O(X^{-A}),
\end{equation}
where $f$ is inert (uniformly in all relevant variables), and is supported on $\frac{\alpha}{X_1 X_2} \asymp \beta X_1$, and $X_2 \asymp \frac{(\alpha \beta)^{1/3}}{\gamma^{2/3}}$.  
Note that this latter estimate can be replaced by $\frac{\alpha}{\gamma} \asymp X_1 X_2^2$, which is more symmetric.
Putting everything together, and simplifying, we obtain \eqref{eq:doubleintegral}.  It is easy to derive the same formula in case $\alpha, \beta, \gamma < 0$ by conjugation.

The final sentence of Lemma \ref{lemma:Iuasymptotic} was proved implicitly along the way, since at least one of $Z_1$ and $Z_2$ will be large under these conditions.
\end{proof}

\section{Analytic properties of $\BesselTransform$}
Our goal in this section is to deduce some useful estimates for $\BesselTransform^{\pm}(x)$ and $\BesselTransform^{\text{holo}}(x)$.  Even more valuable is an integral representation for $\BesselTransform$ that allows us to unify the different cases.

\begin{mylemma}
\label{lemma:HFourier}
Let $\BesselTransform^{+}(x)$ be given by \eqref{eq:H+def}.  Then for $x \ll T$, we have $\BesselTransform^+(x) \ll \Delta x$.  
Furthermore, there exists a function $g$ depending on $\Delta$ and $T$ satisfying $g^{(j)}(x) \ll_{j,A} (1 + |x|)^{-A}$, so that 
\begin{equation}
\label{eq:HFourier}
 \BesselTransform^+(x) = 
 \Delta T \int_{|v| \leq \frac{\Delta^{\varepsilon}}{\Delta}} \cos(x \cosh(v)) e\Big(\frac{vT}{\pi} \Big) g(\Delta v)  dv + O(T^{-A}).
\end{equation}
Furthermore, $\BesselTransform^+(x) \ll T^{-A}$ unless $x \gg \Delta T^{1-\varepsilon}$.  In addition, $\frac{d^k}{dx^k} \BesselTransform^{+}(x)$ is bounded by a polynomial (depending on $k$) of $\Delta, T, x$, and $x^{-1}$.
\end{mylemma}
One can find the asymptotic behavior of $\BesselTransform^+$ given by Jutila and Motohashi \cite[(3.19)]{JM}, which essentially shows that for $x \gg \Delta T^{1-\varepsilon}$ it is of the shape $\frac{\Delta T}{\sqrt{x}} \cos(x + \phi(x,T))$, where $\phi = -\frac{2T^2}{x} + \dots$, with the dots indicating lower-order terms (again, with $x \gg \Delta T^{1-\varepsilon}$).  This comes from asymptotically evaluating the $v$-integral in \eqref{eq:HFourier} by stationary phase.  Our plan is to simply retain this $v$-integral until a later stage, and apply the stationary phase method at the very end.

The bound $\BesselTransform^+(x) \ll \Delta x$ for $x \ll T$ is used only for very large values of $c$, giving a way to initially truncate the sum over $c$.  Note that the error term in  \eqref{eq:HFourier} is problematic for very large values of $c$, since one needs to recover the convergence of the sum of Kloosterman sums.

\begin{proof}
Firstly, we show that for $x \ll T$, $\BesselTransform^+(x) \ll \Delta x$.  For this, we use the integral representation
  \begin{equation}
  \BesselTransform^{+}(x) = \frac{2i}{\pi} \intR \frac{J_{2ir}(x)}{\cosh(\pi r)} h(r) r dr,
 \end{equation}
shift the contour to $\text{Re}(2ir) = 1 + \delta$, $0 < \delta < 1$, and apply the uniform bound (for $t$ real and $x>0$) $J_{1+\delta +2 it}(x) \ll \cosh(\pi t) (\frac{x}{1+|t|})^{1+\delta}$, which in turn follows directly from \cite[8.411.4]{GR}.  This gives
\begin{equation}
 \BesselTransform^{+}(x) \ll x |h(-i/2)| +  \intR \Big(\frac{x}{1+|t|}\Big)^{1+\delta} \Big|h\Big(-\frac{i}{2}-\frac{i\delta}{2} + it\Big)\Big| \cdot \Big|-\frac{i}{2}-\frac{i\delta}{2} + it\Big| dt \ll 
\Delta x,
 \end{equation}
as no other poles are crossed besides $r = -i/2$.
 
 Now we treat the case where $x$ is not extremely small.
 We have (see \cite[8.411.11]{GR})
\begin{equation}
 \frac{J_{2ir}(x) - J_{-2ir}(x)}{\cosh( \pi r)} = \tanh(\pi r) \frac{2}{\pi i} \intR \cos(x \cosh(v)) e\Big(\frac{rv}{\pi} \Big) dv,
\end{equation}
so
\begin{equation}
\label{eq:HintegralRepExact}
\BesselTransform^{+}(x) =  \intR \cos(x \cosh(v)) \frac{4}{\pi^2}\int_0^{\infty} r h(r) \tanh(\pi r) e\Big(\frac{rv}{\pi} \Big) dr dv.
\end{equation}
The inner integral over $r$ is essentially the Fourier transform of a function effectively supported on $r =  T + O(\Delta \log T)$ (outside of this region, $h$ is exponentially small in terms of $T$), so that
\begin{equation}
 \BesselTransform^{+}(x) = \Delta T \intR \cos(x \cosh(v)) e\Big(\frac{vT}{\pi} \Big) g(\Delta v)  dv + O(T^{-A}),
\end{equation}
where $g$ is a function satisfying $g^{(k)}(x) \ll_{k,A} (1+ |x|)^{-A}$.  By the rapid decay of $g$, the contribution to the integral from $|v| \gg \Delta^{-1+\varepsilon}$ is also $O(T^{-A})$, so \eqref{eq:HFourier} holds.  Next we argue that $\BesselTransform^{+}(x) \ll T^{-A}$ unless $x \gg \Delta T^{1-\varepsilon}$.  This follows from Lemma \ref{lemma:exponentialintegral} (repeated integration by parts).  Furthermore, a minor variation of these estimates shows $\frac{d^k}{dx^k} \BesselTransform^{+}(x) \ll_{k,A} T^{-A}$ unless $x \gg \Delta T^{1-\varepsilon}$, for $k=1,2, \dots$.

In order to show $\frac{d^k}{dx^k} \BesselTransform^{+}(x)$ is polynomially bounded by $\Delta, T, x$, and $x^{-1}$, we take the integral representation \eqref{eq:HintegralRepExact} and treat $|v| \leq 1$ and $|v| > 1$ separately.  For $|v| \leq 1$ we simply differentiate inside the integral sign, and use that $\cosh(v) \leq \cosh(1)$ in this range; everything else is polynomially bounded in the other parameters.  For $|v| \geq 1$, we cannot so simply differentiate inside the integral sign because this introduces  powers of $\cosh(v)$ which causes convergence problems.  To get around this, 
we simply first repeatedly integrate by parts to save a large enough power of $e^{\pi |v|}$ to cancel these powers of $\cosh(v)$.  An alternative approach would be to use the recursion formulas for $\frac{d}{dx} J_{\nu}(x)$.
\end{proof}

\begin{mylemma}
\label{lemma:H-Fourier}
Let $\BesselTransform^{-}(x)$ be given by \eqref{eq:H-def}.  Then for $x \ll T$, we have $\BesselTransform^-(x) \ll_{\delta} \Delta T^{\delta} x^{1-\delta}$, for any $0<\delta < 1$.
Furthermore, there exists a function $g$ depending on $\Delta$ and $T$ satisfying $g^{(j)}(x) \ll_{j,A} (1 + |x|)^{-A}$, so that 
\begin{equation}
\label{eq:H-Fourier}
 \BesselTransform^-(x) = 
 \Delta T \int_{|v| \leq \frac{\Delta^{\varepsilon}}{\Delta}} \cos(x \sinh(v)) e\Big(\frac{vT}{\pi} \Big) g(\Delta v)  dv + O(T^{-A}).
\end{equation}
Furthermore, $\BesselTransform^{-}(x) \ll (x+T)^{-A}$ unless $x \asymp T$.  In addition, $\frac{d^k}{dx^k} \BesselTransform^{-}(x)$ is bounded by a polyomial (depending on $k$) in $\Delta, T, x$, and $x^{-1}$.
\end{mylemma}
\begin{proof}
 We quote the relation \cite[8.486.10]{GR}
 \begin{equation}
  K_{\nu}(x) = \frac{x}{2 \nu} K_{\nu + 1}(x) - \frac{x}{2\nu} K_{\nu-1}(x),
 \end{equation}
which implies
\begin{equation}
\label{eq:H-integralRepSplit}
 \BesselTransform^{-}(x) = \frac{x}{i \pi^2} \intR \Big(K_{1+2ir}(x) - K_{1-2ir}(x) \Big) \sinh(\pi r)  h(r) dr.
\end{equation}
Next by \cite[8.432.5]{GR}, we have for $\text{Re}(\nu) \geq -\half$, 
\begin{equation}
 K_{\nu}(x) = \frac{\Gamma(\nu + \frac12) (2/x)^{\nu}}{\Gamma(1/2)} \int_0^{\infty} \frac{\cos(x v)}{(v^2 + 1)^{\nu + \frac12}} dv,
\end{equation}
so that if $\text{Re}(\nu) = \delta$ with $\delta > 0$, we have for $y \in \mr$ by Stirling and a trivial bound
\begin{equation}
\label{eq:KBesselBound}
 |K_{\delta + 2iy}(x)| \ll_{\delta} \frac{(1 + |y|)^{\delta}}{x^{\delta} \cosh(\pi y)}.
\end{equation}
For the part of \eqref{eq:H-integralRepSplit} with $K_{1+2ir}(x)$, we shift the contour to $\text{Re}(1+2ir) = \delta > 0$, and apply \eqref{eq:KBesselBound}.  A similar procedure (shifting the other direction) works for the part with $K_{1-2ir}(x)$, so in all we obtain the bound
\begin{equation}
 \BesselTransform^{-}(x) \ll_{\delta} x^{1-\delta} \intR (1 + |y|)^{ \delta} \exp\Big(- \Big|\frac{y-T}{2\Delta} \Big| \Big) dy \ll_{\delta} \Delta T^{\delta} x^{1-\delta}.
\end{equation}

Next we derive \eqref{eq:H-Fourier}.  For this we use \cite[8.432.4]{GR}, which states
\begin{equation}
  K_{2ir}(x) \cosh(\pi r) = \frac12 \intR \cos(x \sinh v) \cos(2 r v) dv.
\end{equation}
The integral does not converge absolutely, but it does converge uniformly as one can integrate by parts once to save a factor $1/\sinh(v)$.  Thus we derive
\begin{equation}
 \BesselTransform^{-}(x) =  \intR \cos(x \sinh(v)) \frac{1}{\pi^2} \int_0^{\infty} r h(r)  \tanh(\pi r) e\Big(\frac{rv}{\pi} \Big)  dr dv.
\end{equation}
Compare this with \eqref{eq:HintegralRepExact}.  At this point, all the arguments from Lemma \ref{lemma:HFourier} carry over almost without change, so we omit the details.  The main difference is that $\sinh(v) \sim v$ for $v = o(1)$, so by Lemma \ref{lemma:exponentialintegral}, $\BesselTransform^{-}(x)$ is small unless $x \asymp T$.
\end{proof}

Next we consider $\BesselTransform^{\text{holo}}(x)$ defined originally by \eqref{eq:Hholodef}.
By \cite[pp. 85-86]{IwaniecClassical}, we have
\begin{equation}
 \BesselTransform^{\text{holo}}(x) = \frac{T}{4} \intR W(t) c(t) dt,
\end{equation}
where 
\begin{equation}
W(t) = \intR  w\Big(\frac{y-2T}{\Delta}\Big) e(ty) dy 
\end{equation}
and
\begin{equation}
 c(t) = -2i \sin(x \sin(2 \pi t)) + 2i^{-a} \sin(x \cos(2 \pi t)).
\end{equation}
By direct calculation, $W(t) = \Delta e(2t T) \widehat{w}(-\Delta t)$, 
and therefore there exists a function $g$ such that $g^{(j)}(x) \ll (1+|x|)^{-A}$ so that
\begin{equation}
 \BesselTransform^{\text{holo}}(x) = \Delta T \intR e\Big(\frac{vT}{ \pi}\Big) g(\Delta v) (2i \sin(x \sin(v)) + 2i^{-a} \sin(x \cos (v))) dv.
\end{equation}
We can thus decompose $\BesselTransform^{\text{holo}}(x)$ into two pieces that are precisely of the form \eqref{eq:HFourier} and \eqref{eq:H-Fourier} but with $T$ replaced by $T'$, $\cosh(v)$ replaced by $\cos(v)$, and $\sinh(v)$ replaced by $\sin(v)$, and $\cos(x)$ replaced by $\sin(x)$.  In this holomorphic case it is quite easy to show $\BesselTransform^{\text{holo}}(x) \ll Tx$ since $\sin(x) \ll x$.

We close this section by applying these results to the cubic moment problem.
Recall that we wish to show \eqref{eq:quadsum}.  
With $w_1, w_2, w_3$ as in Theorem \ref{thm:quadlinearbound}, 
define
\begin{multline}
\label{eq:Sdef}
S_{\pm}(N_1, N_2, N_3 ;C ;f) 
\\
= \sum_{\substack{ c \asymp C \\ c \equiv 0 \shortmod{q}}} \sum_{n_1, n_2, n_3} \chi_{q}(n_1 n_2 n_3) S(n_1 n_2, \pm n_3;c) w_1(n_1) w_2(n_2) w_3(n_3)  f\Big(\frac{4 \pi \sqrt{n_1 n_2 n_3}}{c} \Big).
 \end{multline}
Theorem \ref{thm:quadlinearbound} amounts to a bound on this $S$ with $f= \BesselTransform$.  Using the weak bound $\BesselTransform(x) \ll T x^{3/4}$, and the Weil bound for Kloosterman sums, we obtain
\begin{equation}
 S_{\pm}(N_1, N_2, N_3 ;C ;\BesselTransform)
 \ll T (N_1 N_2 N_3)^{11/8} (qT)^{\varepsilon} C^{-1/4+\varepsilon},
\end{equation}
which is satisfactory if $C$ is a large power of $qT$, using $N_i \ll (qT)^{1+\varepsilon}$.  Therefore, it suffices to bound $S$ when $C \ll (qT)^{A}$ for some fixed but large $A$.
  For a similar reason, it suffices to bound the terms with $\BesselTransform(x)$ replaced by $\BesselTransform_0(x)$ defined by
\begin{equation}
\label{eq:H0def}
 \BesselTransform_0(x) = \Delta T \int_{|v| \leq \frac{\Delta^{\varepsilon}}{\Delta}} e(x \phi(v)) e\Big(\frac{vT}{\pi} \Big) g(\Delta v)  dv,
\end{equation}
where 
\begin{equation}
  \label{eq:phidef}
  \phi(v) \in \pm \{ \cos(v), \cosh(v), \sin(v), \sinh(v) \}.
 \end{equation}
By this discussion, we have reduced the proofs of \eqref{eq:cubicmomentholmorphicSumOverWeight}, \eqref{eq:cubicmomentMaass}, and \eqref{eq:cubicmomentEisenstein} (and hence, Theorem \ref{thm:cubicmoment}) to the following 
\begin{myprop}
\label{prop:generalH}
 Let $\BesselTransform_0$ be a function of the form \eqref{eq:H0def},
 where $g$ is a function satisfying $g^{(j)}(x) \ll_{j,A} (1 + |x|)^{-A}$.
 Then for $N_i \ll (qT)^{1+\varepsilon}$ and $C \ll (qT)^{A}$ with some fixed $A$, we have
 \begin{equation}
\label{eq:goalbound}
 S(N_1, N_2, N_3 ;C ;\BesselTransform_0) \ll (N_1 N_2 N_3)^{1/2} \Delta T (qT)^{\varepsilon}.
 \end{equation}
\end{myprop}

\section{Analytic separation of variables}
\label{section:separationAnalytic}

Following the method of \cite{CI}, we begin by
applying Poisson summation to each of $n_1, n_2, n_3$ modulo $c$ in \eqref{eq:Sdef}.  That is,
\begin{equation}
\label{eq:SGK}
S_{\pm}(N_1, N_2, N_3 ;c ;\BesselTransform_0) = \sum_{m_1, m_2, m_3 \in \mz} G_{\pm}(m_1, m_2, m_3;c) K(m_1, m_2, m_3, c),
\end{equation}
where
\begin{equation}
\label{eq:Gpmdef}
G_{\pm}(m_1 ,m_2, m_3;c) = c^{-3} \sum_{a_1, a_2, a_3 \shortmod{c}} \chi_{q}(a_1 a_2 a_3) S(a_1 a_2, \pm a_3;c) e_c(a_1m_1 + a_2 m_2 + a_3 m_3),
\end{equation}
and
\begin{equation}
\label{eq:Kdef}
K(m_1 ,m_2, m_3, c) = \int_{\mr^3} w_1(t_1) w_2(t_2) w_3(t_3) \BesselTransform_0\Big(\frac{4 \pi \sqrt{t_1 t_2 t_3}}{c} \Big) e_c(-m_1 t_1 - m_2 t_2 - m_3 t_3) dt_1 dt_2 dt_3.
\end{equation}
By changing variables $a_3 \rightarrow -a_3$, one derives $G_{-}(m_1, m_2, m_3;c) = \chi_q(-1) G_{+}(m_1, m_2, -m_3;c)$, which will allow us to mainly focus on the $+$ sign case.  Let $G = G_{+}$ as shorthand. 

The sum $G$ was studied extensively by Conrey and Iwaniec \cite{CI}, so we will quote their results.  Our work differs from \cite{CI} in the nature of the weight function $K$.  Conrey and Iwaniec showed, for $T$ essentially bounded, that $K$ has a phase $e_c(-m_1 m_2 m_3)$, and is small except for $m_i \ll (N_1 N_2 N_3)^{1/2}/N_i$.  The difficuly in extending their work is that for $T$ larger, one needs to more carefully treat the lower order terms in the phase of $K$.  More generally, what we show is that $K$ still has a phase as mentioned above, but has a lower-order phase that still contains the factor $m_1 m_2 m_3/c$ in a block.  This still allows us to use a Mellin technique to separate these four variables with a single integral.  The important properties of $K$ are summarized in the following
\begin{mylemma}
\label{lemma:Kbound}
Suppose that $|m_i| \asymp M_i$ for each $i$, and $c \asymp C$.  
\begin{enumerate}
\item Suppose $\phi(v)$ is $\pm \cos v$ or $\pm \cosh v$.  Then
\begin{multline}
\label{eq:Kseparated}
K(m_1, m_2, m_3, c) = \frac{C^{3/2} \Delta T (N_1 N_2 N_3)^{1/2} e_c(-m_1 m_2 m_3)}{(M_1 M_2 M_3)^{1/2}}  L(m_1, m_2, m_3, c)
\\
+ O(T^{-A} \prod_{i=1}^{3} (1+ |m_i|)^{-2} ),
\end{multline}
where $L$ is a function that takes the form
\begin{equation}
\label{eq:Rdef}
 L(m_1, m_2, m_3, c) = \frac{1}{V} \int_{|{\bf y}| \leq T^{\varepsilon}} m_1^{iy_1} m_2^{i y_2} m_3^{i y_3} c^{i y_4} \int_{|u| \ll U} \ell(u, {\bf y}) \Big(\frac{m_1 m_2 m_3}{c}\Big)^{iu} du d {\bf y}, 
\end{equation}
where ${\bf y} = (y_1, y_2, y_3, y_4)$, $V = T$, and 
\begin{equation}
\label{eq:Udef} 
 U = \frac{T^2 C}{(N_1 N_2 N_3)^{1/2}}.
\end{equation}
Here $\ell(u, {\bf y}) \ll 1$ does not depend on $c$ and the $m_i$.  Furthermore, $L$ vanishes unless
\begin{equation}
\label{eq:CsizeCos}
 C \ll \frac{(N_1 N_2 N_3)^{1/2}}{\Delta^{1-\varepsilon} T}, \quad \text{and} \quad M_i \asymp \frac{(N_1 N_2 N_3)^{1/2}}{N_i}, i=1,2,3.
\end{equation}
\item Suppose $\phi(v)$ is $\pm \sin v $ or $\pm \sinh v$.  Then
\begin{multline}
\label{eq:Kseparated2}
K(m_1, m_2, m_3, c) = \frac{C^{3/2} \Delta T (N_1 N_2 N_3)^{1/2} e_c(m_1 m_2 m_3)}{(M_1 M_2 M_3)^{1/2}}  L(m_1, m_2, m_3, c)
\\
+ O(T^{-A} \prod_{i=1}^{3} (1+ |m_i|)^{-2} ),
\end{multline}
with the following parameters.
If $\frac{M_i N_i}{C} \gg T^{\varepsilon}$ for some $i$, then $L$ is defined as \eqref{eq:Rdef} but with $V = T$, $U = X^{1/3} T^{2/3}$, and $X = \frac{M_1 M_2 M_3}{C}$.  In addition, $L$ vanishes unless 
\begin{equation}
\label{eq:CMisizesforSin}
 C \asymp \frac{(N_1 N_2 N_3)^{1/2}}{T}, \quad  \quad M_i \ll \frac{(N_1 N_2 N_3)^{1/2}}{ N_i \Delta^{1-\varepsilon} },
\end{equation}
and $M_1 N_1 \asymp M_2 N_2 \asymp M_3 N_3$.
If $\frac{M_i N_i}{C} \ll T^{\varepsilon}$ for all $i$ then $K$ has the same form as 
\eqref{eq:Kseparated} with $L$ defined as \eqref{eq:Rdef} but with $V = T^{-\varepsilon} X^{-1/2}$, $U = T^{\varepsilon}$, and $X = \frac{M_1 M_2 M_3}{C}$.
\end{enumerate}
\end{mylemma}
Remarks.
Here the ${\bf y}$-integral is practically harmless, so in effect Lemma \ref{lemma:Kbound} expresses $K$ in terms of an integral of length $U$, and has all variables $m_1, m_2, m_3, c$ separated.  The expressions \eqref{eq:Kseparated} and \eqref{eq:Kseparated2} are identical except for the sign in $e_c(\mp m_1 m_2 m_3)$; in both cases, this phase cancels a factor coming from the calculation of $G_{\pm}(m_1, m_2, m_3 ;c)$ defined by \eqref{eq:Gpmdef} below.  Actually, the appearance of $e_c(m_1 m_2 m_3)$ in \eqref{eq:Kseparated2} is artificial in the sense that this phase is of lower order than the phase implicit in $L(m_1, m_2, m_3, c)$ and we could have equally well showed that $e_c(m_1 m_2 m_3) L(m_1, m_2, m_3, c)$ has a representation in the form \eqref{eq:Rdef}.  

A careful reader may notice that we have not stated an expression for $K$ that is a true analog of \eqref{eq:KapproxSketchSection}.  The reason for this omission is that for the large sieve method, it is desirable to have the variables $m_1, m_2, m_3$, and $c$ separated in multiplicative (Mellin) form, and so developing an asymptotic expansion of the new phase of $L$ is at odds with this goal.  It should be possible to develop the correct form of \eqref{eq:KapproxSketchSection} by applying the stationary phase method to \eqref{eq:Phiintegralpm}.  Instead, the key idea is that in \eqref{eq:Kformula2}, the variables $m_1, m_2, m_3, c$ only occur in the block form $m_1 m_2 m_3/c$.  It then suffices to understand the Mellin transform of $\Phi$ in a somewhat crude form (upper bounds suffice)
 which is a bit easier than forming an asymptotic evaluation with many lower-order terms.

\begin{proof}
As our first step, we integrate by parts three times in each of the $t_i$ for which $m_i \neq 0$, allowing us to obtain a crude bound of the form
\begin{equation}
 K(m_1, m_2, m_3, c) \ll P(T, \Delta, N_1, N_2, N_3, c) \prod_{i=1}^{3} (1+ |m_i|)^{-3},
\end{equation}
where $P$ is some fixed polyomial.  This bound is sufficient for Lemma \ref{lemma:Kbound} when some $m_i$ is $\gg T^{A'}$ for some large $A'$ depending polynomially on $A$.

For the rest of the proof, suppose that $|m_i| \ll T^{A'}$ for some $A'$, and each $i$.  Then 
in the expansion for $\BesselTransform(x)$ we may assume $x \gg T$, since otherwise $\BesselTransform_0$ is extremely small, and we obtain the desired bound for Lemma \ref{lemma:Kbound}.  

Before we jump into more intricate analysis, we can fix the sizes of $C$ as stated
in \eqref{eq:CsizeCos} and \eqref{eq:CMisizesforSin} using $x \asymp \frac{(N_1 N_2 N_3)^{1/2}}{C}$, and Lemmas \ref{lemma:HFourier} and \ref{lemma:H-Fourier}.

It is a bit awkward to directly treat $K$ by stationary phase in each $t_i$, so instead we use the following workaround.
By the change of variables $t_3 = u/(t_1 t_2)$, we obtain
\begin{equation}
\label{eq:Kformula0}
K(m_1, m_2, m_3, c) = \int_{0}^{\infty} \BesselTransform_0\Big(\frac{4 \pi \sqrt{u}}{c} \Big) I(u) du + O(T^{-A}),
\end{equation}
where
\begin{equation}
I(u) = \int_0^{\infty} \int_0^{\infty} w_3\Big(\frac{u}{t_1 t_2} \Big) w_1(t_1) w_2(t_2)  e_c(-\frac{m_3 u}{t_1 t_2} - m_1 t_1 - m_2 t_2) \frac{dt_1 dt_2}{t_1 t_2}.
\end{equation}
The asymptotic expansion of $I(u)$ will be derived from Lemma \ref{lemma:Iuasymptotic}.  This saves
the more difficult case of stationary phase with $\BesselTransform_0$ for last when the integral is one-dimensional.  This feature pleasantly allows us to unify cases for as long as possible since we do not use any properties of $\BesselTransform_0$ until the later stages.

Suppose that
\begin{equation}
\label{eq:milowerbound}
 \frac{M_i N_i}{C} \gg T^{\varepsilon}, \quad i=1,2,3,
\end{equation}
so that the integrand defining $I(u)$ is oscillatory.  The opposite case is somewhat easier and we will return to it later.
Lemma \ref{lemma:Iuasymptotic} then shows that 
\begin{equation}
I(u) = \frac{c}{(m_1 m_2 N_1 N_2)^{1/2}} e\Big(-\frac{3 (m_1 m_2 m_3 u)^{1/3}}{c}\Big) w_4(u) + O(T^{-A}),
\end{equation}
where $w_4$ is a smooth function supported on $u \asymp N_1 N_2 N_3$, depending on $m_1, m_2, m_3, c$, etc., but that is inert in all variables.  Furthermore, we derive from \eqref{eq:alphabetagammasupport} that $w_4$ vanishes unless
\begin{equation}
\label{eq:misymmetrysizes}
m_1 N_1 \asymp m_2 N_2 \asymp m_3 N_3.
\end{equation}
Hence we derive
\begin{multline}
\label{eq:Kformula}
K = 
\Big(\frac{c \Delta T}{(m_1 m_2 N_1 N_2)^{1/2}}  \int_{|v| \leq \frac{\Delta^{\varepsilon}}{\Delta}}  e\Big(\frac{vT}{\pi} \Big) g(\Delta v) 
\\
\int_0^{\infty} e\Big(\frac{2\sqrt{u}}{c} \phi(v) 
- 
\frac{3(m_1 m_2 m_3 u)^{1/3}}{c}\Big) w_4(u) du  dv \Big)
+ O(T^{-A}).
\end{multline}

The $u$-integral can also by analyzed by stationary phase, as \eqref{eq:milowerbound} and Part 1 of Lemma \ref{lemma:exponentialintegral} shows that it is small unless a stationary point exists, which implies 
\begin{equation}
\label{eq:phisupport}
|\phi(v)| \asymp |m_1 m_2 m_3|^{1/3} (N_1 N_2 N_3)^{-1/6},
\end{equation}
(additionally,  $\phi(v)$ must have the same sign as $m_1 m_2 m_3$).  Note that if $\phi(v) = \pm \cos v$ or $\pm \cosh v$ then $|\phi(v)| \asymp 1$ so this implies 
 $|m_1 m_2 m_3| \asymp 
 (N_1 N_2 N_3)^{1/2}$, which by \eqref{eq:misymmetrysizes} leads to the assumption $M_i \asymp (N_1 N_2 N_3)^{1/2}/N_i$ in Part 1 of Lemma \ref{lemma:Kbound}.

 However, if $\phi(v) = \pm \sin v$ or $\pm \sinh v$, then we obtain $|v| \asymp |m_1 m_2 m_3|^{1/3} (N_1 N_2 N_3)^{-1/6}$, and since $|v| \leq \Delta^{-1 + \varepsilon}$, we conclude that $|m_1 m_2 m_3| \ll (N_1 N_2 N_3)^{1/2} \Delta^{-3+\varepsilon}$.  Again using \eqref{eq:misymmetrysizes} leads to the upper bound on $M_i$ in \eqref{eq:CMisizesforSin}.

Assuming \eqref{eq:phisupport}, the stationary point at $u_0 = (m_1 m_2 m_3)^2/\phi(v)^6$ potentially lies inside the support of $w_4$, so
\begin{multline}
\label{eq:uintegral}
 \int_0^{\infty} e\Big(\frac{2\sqrt{u}}{c} \phi(v) 
 - \frac{3(m_1 m_2 m_3 u)^{1/3}}{c}\Big) w_4(u) du 
 \\
 = \frac{c^{1/2} (N_1 N_2 N_3)^{5/6}}{|m_1 m_2 m_3|^{1/6}} e\Big(-\frac{m_1 m_2 m_3}{c \phi(v)^2}\Big) w_5(v) + O(T^{-A}),
\end{multline}
where $w_5$ is inert in terms of $v$ as well as the $m_i$ and $c$, and $w_5$ has support on \eqref{eq:phisupport}.  The fact that $w_5$ is inert in terms of $v$ perhaps requires some discussion.  We naturally obtain an inert function in terms of $\phi(v)$, but since $\phi(v)$ has bounded derivatives for $|v| \leq 1$, we do in fact obtain an inert function of $v$.

By inserting \eqref{eq:uintegral} into \eqref{eq:Kformula}, and using \eqref{eq:misymmetrysizes} to obtain a more symmetric expression,
we derive
\begin{equation}
\label{eq:Kformula2}
K(m_1, m_2, m_3, c) = \frac{c^{3/2} \Delta T (N_1 N_2 N_3)^{1/2}}{|m_1 m_2 m_3|^{1/2} } 
e\Big(\mp \frac{m_1 m_2 m_3}{c} \Big) \Phi\Big(\frac{m_1 m_2 m_3}{c} \Big)
+ O(T^{-A}),
\end{equation}
where
\begin{equation}
\label{eq:Phiintegralpm}
 \Phi(x) = \int_{|v| \leq \frac{\Delta^{\varepsilon}}{\Delta}}  e\Big(\frac{vT}{\pi} \Big) g(\Delta v) 
 e\Big(x(\pm 1-\phi(v)^{-2} )\Big) w_6(v)  dv,
\end{equation}
and $w_6$ is another inert function having the same properties as $w_5$.  Here the $\mp$ sign in \eqref{eq:Kformula2} is $-$ for the $\cos v$ or $\cosh v$ cases, and $+$ for the $\sin v$ or $\sinh v$ cases, and the sign in \eqref{eq:Phiintegralpm} respects these.

Finally we shall use the Mellin technique to analyze $\Phi(x)$.
\begin{mylemma}
\label{lemma:PhiMellin}
 Suppose that $g$ is a function satisfying $g^{(j)}(x) \ll_{j,A}(1 + |x|)^{-A}$, $\phi(v)$ is given via \eqref{eq:phidef}, and $\Phi$ is defined by \eqref{eq:Phiintegralpm} for some inert function $w_6$.  Then for $x \asymp X \gg 1$, we have
 \begin{equation}
  \Phi(x) = \frac{1}{T} \int_{|t| \ll U} \lambda_{X,T}(t) x^{it} dt + O(T^{-A}),
 \end{equation}
where $\lambda_{X,T}$ and $U$ depend on $X$, $T$, and the choice of $\phi(v)$, and satisfy $\lambda_{X,T}(t) \ll 1$ and 
\begin{equation}
\label{eq:Uandlambda}
  \begin{cases}
                U = T^2/X 
                &\phi(v) = \pm \cosh(v) \text{ or } \pm \cos(v) \\
               U = X^{1/3} T^{2/3} 
               \quad 
               &\phi(v) =  \pm \sinh(v) \text{ or } \pm \sin(v).
                      \end{cases}
\end{equation}
\end{mylemma}

\begin{proof}[Proof of Lemma \ref{lemma:PhiMellin}]
Suppose that $1\leq Y < y < 2Y$, and let $w(y)$ be a smooth function such that $w(y) = 1$ on this interval, and $w(y) = 0$ for $y < Y/2$ and $y > 3Y$.  By the Mellin inversion formula, for $Y < y < 2Y$, we have
\begin{equation}
e(y) = w(y) e(y) = \intR f(t) y^{it} dt, \qquad f(t) = \frac{1}{2\pi} \int_0^{\infty} w(y) e(y) y^{-it} \frac{dy}{y}.
\end{equation}
Integration by parts shows $f(t) \ll (|t| + Y)^{-A}$ unless $t \asymp Y$, in which case $f(t) \ll Y^{-1/2}$, by stationary phase.  
We sometimes write $f = f_Y$ to help us remember the parameter associated to $f$.

We will use this in $\Phi(x)$, however we need to treat the two types of $\phi$ a bit differently.  If $\phi(v) = \pm \sin v$ or $\pm \sinh v$, we have $|\phi(v)| \asymp |v| = o(1)$, so $x|1+\phi(v)^{-2}| \asymp X/|v|^2$, and thus letting $|v_0| = |m_1 m_2 m_3|^{1/3} (N_1 N_2 N_3)^{-1/6}$ (recall $v$ is supported for $|v| \asymp |v_0|$), we have
\begin{equation}
 \Phi(x) = \int_{|t| \asymp X/|v_0|^2} f_{X/v_0^2}(t) x^{it} \Big(\int_{|v| \leq \frac{\Delta^{\varepsilon}}{\Delta}}  e\Big(\frac{vT}{\pi} \Big) g(\Delta v) 
  (\phi(v)^{-2}+1)^{it}   w_6(v)  dv \Big) dt + O(T^{-A}).
\end{equation}
We need to bound the inner integral over $v$.  
By a Taylor expansion, we have
\begin{equation}
 (\phi(v)^{-2}+1)^{it} = e^{-2it \log v + it(d_2 v^2 + d_4 v^4 + \dots)}, 
\end{equation}
for certain constants $d_i$ (depending only on the choice of $\sin v$ or $\sinh v$)
so the $v$-integral is small except near the stationary point $|v| \asymp (t/T)$.  Note that we have $|v_0| \asymp \frac{|t|}{T} \asymp \frac{X}{v_0^2 T}$, so that $|v_0| \asymp (x/T)^{1/3}$.  As a pleasant consistency check, note that $x \asymp \frac{M_1 M_2 M_3}{C}$, and so $(x/T)^{1/3} \asymp (M_1 M_2 M_3)^{1/3} (N_1 N_2 N_3)^{-1/6}$, using \eqref{eq:CMisizesforSin}.  The length of the $t$-integral is therefore seen to be $X/|v_0|^2 \asymp X^{1/3} T^{2/3}$, consistent with the claimed size of $U$ in the second line of \eqref{eq:Uandlambda}.
In the inner $v$-integral, the second derivative of the phase is of size $t/v_0^2$, 
so we have
\begin{equation}
 \int_{|v| \leq \frac{\Delta^{\varepsilon}}{\Delta}}  e\Big(\frac{vT}{\pi} \Big) g(\Delta v) 
  (1+\phi(v)^{-2})^{it}   w_6(v)  dv \ll \frac{|v_0|}{\sqrt{|t|}},
\end{equation}
and therefore $f_{X/v_0^2}(t) \frac{|v_0|}{\sqrt{|t|}}
 \ll \frac{|v_0|^2}{\sqrt{X} \sqrt{|t|}} \asymp \frac{|v_0|^3}{X} \asymp T^{-1}$, leading to the bound on $\lambda_{X,T}(t)$
in the second line of \eqref{eq:Uandlambda}.

Next we treat $\phi(v) = \pm \cos v$ or $\pm \cosh v$.  
Here we have, for certain constants $d_i'$, 
\begin{equation}
1-\phi(v)^{-2} = \frac12 v^2 +  d_4' v^4 + d_6' v^6 + \dots,
\end{equation}
so in this case we may initially restrict the $v$-integral so that $v \asymp T/x$ (prior to aplying the $t$-integral formula for $e(y)$), using Lemma \ref{lemma:exponentialintegral} again.  In this range, we have $x(1-\phi(v)^{-2}) \asymp T^2/X$, so that
\begin{equation}
 \Phi(x) = \int_{|t| \asymp T^2/X} f_{T^2/X}(t) x^{it} \Big(\int_{|v| \asymp \frac{T}{X}}  e\Big(\frac{vT}{\pi} \Big) g(\Delta v) 
  (1-\phi(v)^{-2})^{it}   w_6(v)  dv \Big) dt + O(T^{-A}).
\end{equation}
Now we have
\begin{equation}
  (1-\phi(v)^{-2})^{it} = e^{2it \log{v} + it \log(1 + c_4 v^2 + c_6 v^4 + \dots)}.
\end{equation}
The second derivative of the phase is of size $t/v^2 \asymp X$.  Hence
\begin{equation}
 \int_{|v| \leq \frac{\Delta^{\varepsilon}}{\Delta}}  e\Big(\frac{vT}{\pi} \Big) g(\Delta v) 
  (1-\phi(v)^{-2})^{it}   w_6(v)  dv \ll X^{-1/2},
\end{equation}
and $f_{T^2/X}(t) X^{-1/2} \ll T^{-1}$.
\end{proof}

Now we apply Lemma \ref{lemma:PhiMellin} to \eqref{eq:Kformula2}, where recall that we restrict to $c \asymp C$, and $|m_i| \asymp M_i$.  Then we have $X = C^{-1} M_1 M_2 M_3$, and so in the $\pm \cos v$ or $\pm \cosh v$ cases, we have $U = \frac{T^2}{X} = \frac{T^2 C}{M_1 M_2 M_3} \asymp \frac{T^2 C}{(N_1 N_2 N_3)^{1/2}}$, so we have shown \eqref{eq:Kseparated}.  The $\pm \sin v$ or $\pm \sinh v$ cases are similar.

Finally, we consider the case where 
\begin{equation}
\label{eq:MiUB}
\frac{M_i N_i}{C} \ll T^{\varepsilon}.  
\end{equation}
In the $\pm \cos v$ or $\pm \cosh v$ cases, we claim that $K$ is very small.  To see this, we go back to the original definition \eqref{eq:Kdef}, combined with \eqref{eq:H0def}.  We also recall that $\BesselTransform(x)$ is small unless $x \gg \Delta T^{1-\varepsilon}$, which means that
each $t_i$-integral is oscillatory.  Therefore, Lemma \ref{lemma:exponentialintegral} shows that $K$ is small unless $M_i \asymp (N_1 N_2 N_3)^{1/2}/N_i$; in other words, the inequalities in \eqref{eq:CsizeCos} remain valid.  
But then,
\begin{equation}
 \frac{M_i N_i}{C} \asymp \frac{(N_1 N_2 N_3)^{1/2}}{C} \gg \Delta T^{1-\varepsilon},
\end{equation}
which is not consistent with \eqref{eq:MiUB} (meaning, $K$ is small).

Next we study the $\pm \sin v$ or $\pm \sinh v$ cases.  Again we use \eqref{eq:Kdef} and \eqref{eq:H0def}, and change
variables $t_i \rightarrow N_i t_i$, to obtain
\begin{multline}
\label{eq:Kformula3}
K = \Delta T (N_1 N_2 N_3) \int_{\mr^3} \Big(\prod_{i=1}^{3} w_i(N_i t_i) e_c(-m_i N_i t_i) \Big) 
 \\
 \int_{|v| \leq \frac{\Delta^{\varepsilon}}{\Delta}} 
 e\Big( \frac{2 \sqrt{N_1 N_2 N_3}}{c} \sqrt{t_1 t_2 t_3} \phi(v)\Big) e\Big(\frac{vT}{\pi} \Big) g(\Delta v)  dv dt_1 dt_2 dt_3 + O(T^{-A}).
\end{multline}
The $v$-integral is small unless $\frac{\sqrt{N_1 N_2 N_3}}{C} \asymp T$.  The phase of each $t_i$ integral is then of size $T |v| + O(T^{\varepsilon})$, using \eqref{eq:MiUB}.  Thus if $|v| \gg T^{-1+\varepsilon}$, then repeated integration by parts shows $K$ is small, so in \eqref{eq:Kformula3} we may shorten the $v$-integral by assuming $|v| \ll T^{-1+\varepsilon}$, without creating a new error term.  Next we artificially multiply $K$ by $e_c(-m_1 m_2 m_3) e_c(m_1 m_2 m_3)$, and use a Mellin transform to separate the variables in $e_c(-m_1 m_2 m_3)$ (we keep the plus sign part as it is, since it is visible in \eqref{eq:Kseparated2}).  The cost in doing this is an integral of length $O(T^{\varepsilon})$, by \eqref{eq:MiUB}.  Thus in all we obtain an expression of the form
\begin{equation}
 K = \Delta  (N_1 N_2 N_3) e_c(m_1 m_2 m_3) L(m_1, m_2, m_3, c)  + O(T^{-A}),
\end{equation}
where $L$ is of the form \eqref{eq:Rdef}, with $V = T^{-\varepsilon}$, and $U = T^{\varepsilon}$.  To put this into the form of \eqref{eq:Kseparated}, we use
\begin{equation}
 \Delta  N_1 N_2 N_3 \asymp \Delta T (N_1 N_2 N_3)^{1/2} C = \frac{C^{3/2} \Delta T (N_1 N_2 N_3)^{1/2}}{(M_1 M_2 M_3)^{1/2}} \Big(\frac{M_1 M_2 M_3}{C}\Big)^{1/2},
\end{equation}
which is now of the desired form by absorbing $X^{-1/2}$ into the $V$ appearing in \eqref{eq:Rdef}.
\end{proof}

\section{Arithmetic separation of variables and the large sieve}
\label{section:largesieve}
The material in this section is logically independent of Section \ref{section:separationAnalytic}.  

According to Lemma 10.2 of \cite{CI}, we have an evaluation of $G$ as follows.  Let $c = qr$ with $q$ squarefree, and suppose
\begin{equation}
\label{eq:coprimality}
(m_3, r) = 1, \quad \text{and} \quad (m_1 m_2, q, r) = 1.  
\end{equation}
Then
\begin{equation}
\label{eq:GH}
 G(m_1, m_2, m_3;c) =  e_c(m_1 m_2 m_3)  \frac{h \chi_{kl}(-1)}{r q^2 \phi(k)} R_k(m_1) R_k(m_2) R_k(m_3) H(\overline{rhk} m_1 m_2 m_3;l),
\end{equation}
where $h=(r,q)$, $k=(m_1 m_2 m_3, q)$, $l= q/(hk)$, and $R_k(m) = S(0,m;k)$ is the Ramanujan sum.  If the coprimality conditions \eqref{eq:coprimality} are not satisfied then $G$ vanishes.  Here
\begin{equation}
\label{eq:Hwqdef}
 H(w;q) = \sum_{u, v \shortmod{q}} \chi_{q}(uv(u+1)(v+1)) e_q((uv-1)w).
\end{equation}
Also see \cite[(11.7)]{CI}, giving
\begin{equation}
\label{eq:HH*}
H(w;q) = \sum_{q_1 q_2 = q} \mu(q_1) \chi_{q_1}(-1) H^*(\overline{q_1} w;q_2),
\end{equation}
where
\begin{equation}
 H^*(w;q) = \sum_{\substack{u, v \shortmod{q} \\ (uv-1, q) =1}} \chi_{q}(uv(u+1)(v+1)) e_q((uv-1)w),
\end{equation}
and from \cite[(11.9)]{CI}, 
\begin{equation}
\label{eq:H*separationofvariables}
 H^*(w;q) = \frac{1}{\phi(q)} \sum_{\psi \shortmod{q}} \tau(\overline{\psi}) g(\chi, \psi) \psi(w).
\end{equation}

\begin{mylemma}[Hybrid large sieve]
\label{lemma:largesieve}
 Suppose $U \geq 1$, and let $a_n$ be a sequence of complex numbers.  Then
 \begin{equation}
  \int_{-U}^{U} \sum_{\psi \shortmod{q}} \Big| \sum_{n \leq N} a_n \psi(n) n^{iu} \Big|^2 du \ll (qU + N) \sum_{n \leq N} |a_n|^2.
 \end{equation}
\end{mylemma}
This is Theorem 2 of \cite{Gallagher}.

\begin{mylemma}
\label{lemma:bilinearGeneral}
 Suppose that $q$ is odd and squarefree, and let $\alpha_{m_1}, \beta_{m_2}, \gamma_{m_3}, \delta_r$ be any sequence of complex numbers.  Suppose $(b_1 b_2,q) = 1$, $a$ is a nonzero real number, and $U \geq 1$.  Then we have
  \begin{multline}
 \label{eq:bilinearGeneral}
  \int_{|u| \leq U} \Big| \sum_{\substack{m_1, m_2, m_3 \\ m_i \asymp M_i}} \sum_{r \ll R} \alpha_{m_1} \beta_{m_2} \gamma_{m_3} \delta_r H^*(\overline{b_1r} b_2  m_1 m_2 m_3; q) \Big(\frac{m_1 m_2 m_3}{ar} \Big)^{iu}  \Big| du
 \\
 \ll 
 q^{1/2 + \varepsilon} (qU + M_1 M_2)^{1/2} (qU + M_3 R)^{1/2} \Big(\sum_{m_1, m_2, m_3, r} |\alpha_{m_1} \beta_{m_2} \gamma_{m_3} \delta_r|^2 \Big)^{1/2}.
 \end{multline}
 The implied constant depends only on $\varepsilon$.  Furthermore, \eqref{eq:bilinearGeneral} holds with $H^*$ replaced by $H$.  In addition, if we restrict the sums over $m_3$ and $r$ so that $(m_3, r) = 1$, then the left hand side of \eqref{eq:bilinearGeneral} (with either $H^*$ or $H$) is
 \begin{equation}
 \label{eq:bilinearGeneralCoprime}
  \ll q^{1/2 + \varepsilon} (qU + M_1 M_2)^{1/2} (qU + M_3 R)^{1/2} \Big(\sum_{d, m_1, m_2, m_3, r} d^{1+\varepsilon} |\alpha_{m_1} \beta_{m_2} \gamma_{dm_3} \delta_{dr}|^2 \Big)^{1/2}
 \end{equation}

 \end{mylemma}
 This is a variation on Lemma 11.1 of \cite{CI} where the main difference is that here we have an integral which detects orthogonality in the archimedean aspect.
\begin{proof}
 By \eqref{eq:H*separationofvariables}, the left hand side of \eqref{eq:bilinearGeneral} equals
 \begin{equation}
 \label{eq:someequationarisinginaproof2}
  \int_{|u| \leq U} \Big| \frac{1}{\phi(q)} \sum_{\psi \shortmod{q}} \sum_{\substack{m_1, m_2, m_3 \\ m_i \asymp M_i}} \sum_{r \ll R} \alpha_{m_1} \beta_{m_2} \gamma_{m_3} \delta_r  \tau(\overline{\psi}) g(\chi, \psi) \psi(\overline{b_1 r} b_2 m_1 m_2 m_3) \Big(\frac{m_1 m_2 m_3}{ar} \Big)^{iu}  \Big| du,
 \end{equation}
which we arrange into a bilinear form by grouping together $m_1$ and $m_2$ as well as $m_3$ and $r$, giving that \eqref{eq:someequationarisinginaproof2} is bounded by
\begin{equation}
\label{eq:someequationarisinginaproof}
  \int_{|u| \leq U}  \sum_{\psi \shortmod{q}} \frac{| g(\chi, \psi) \tau(\overline{\psi})|}{\phi(q)} \Big|  \sum_{m_1, m_2} \alpha_{m_1} \beta_{m_2} \psi( m_1 m_2) (m_1 m_2)^{iu} \Big| \cdot \Big|\sum_{m_3, r} \gamma_{m_3}  \delta_r \psi(\overline{r} m_3) \Big(\frac{m_3 }{r} \Big)^{iu}\Big|   du.
\end{equation}
Conrey and Iwaniec showed $|g(\chi, \psi)| \ll q^{1+\varepsilon}$, and of course $|\tau(\overline{\psi})| \leq q^{1/2}$ (e.g., see Lemma 3.1 of \cite{IK}).  By Cauchy's inequality, and some rearrangements, we obtain that this is
\begin{equation}
 \ll q^{\frac12+\varepsilon} \Big(\int_{|u| \leq U} \sum_{\psi \shortmod{q}} \Big| \sum_{n \ll M_1 M_2} a_n \psi(n) n^{iu} \Big|^2 du\Big)^{\frac12} \Big(\int_{|u| \leq U} \sum_{\psi \shortmod{q}} \Big| \sum_{n \ll M_3 R} b_n \psi(n) n^{iu} \Big|^2 du \Big)^{\frac12},
\end{equation}
where $(a_n)$ is the Dirichlet convolution of $(\alpha_{m_1})$ and $(\beta_{m_2})$, and likewise for $(b_n)$.  By Lemma \ref{lemma:largesieve}, we obtain the bound \eqref{eq:bilinearGeneral},  as desired.

The case with $(m_3, r) = 1$ follows similar lines.  In \eqref{eq:someequationarisinginaproof} we use M\"{o}bius inversion, writing $\sum_{d|(m_3, r)} \mu(d)$ to detect this condition.  Then we move the sum over $d$ to the outside, and apply Cauchy's inequality, in the form $|\sum_{d} a_d|^2 \leq  \zeta(1+\varepsilon) \sum_{d} d^{1+\varepsilon} |a_d|^2$.  The remaining steps are identical to the previous case.

To replace $H^*$ by $H$, we use \eqref{eq:HH*} and apply \eqref{eq:bilinearGeneral} or \eqref{eq:bilinearGeneralCoprime}, whichever is appropriate.
\end{proof}

 \begin{mylemma}
 \label{lemma:bilinearWithG}
  Let conditions be as in Lemma \ref{lemma:bilinearGeneral}.  Then
 \begin{multline}
 \label{eq:bilinearWithG}
  \int_{|u| \leq U} \Big|\sum_{\substack{m_1, m_2, m_3 \\ m_i \asymp M_i}} \sum_{r \asymp R} \alpha_{m_1} \beta_{m_2} \gamma_{m_3} \delta_r G(m_1, m_2, m_3, qr) e_{-qr}(m_1 m_2 m_3) \Big(\frac{m_1 m_2 m_3}{qr} \Big)^{iu}  \Big| du
 \\
 \ll 
 \frac{q^{1/2 + \varepsilon}}{R q^2} (qU + M_1 M_2)^{1/2} (qU + M_3 R)^{1/2} \Big(\sum_{d, m_1, m_2, m_3, r} d^{1+\varepsilon} |\alpha_{m_1} \beta_{m_2} \gamma_{d m_3} \delta_{dr}|^2 \Big)^{1/2}.
 \end{multline}
 \end{mylemma}
 Remark: If $\gamma_{m_3} \ll 1$ and $\delta_{r} \ll 1$, which is all that we require in our use of Lemma \ref{lemma:bilinearWithG}, then the extra sum over $d$ does not change the bound arising from $d=1$.
\begin{proof}
By \eqref{eq:GH}, we obtain that the left hand side of \eqref{eq:bilinearWithG} is of the form
\begin{multline}
 \int_{|u| \leq U} \Big|\sumstar_{\substack{m_1, m_2, m_3 \\ m_i \asymp M_i}} \sum_{r \asymp R} \alpha_{m_1} \beta_{m_2} \gamma_{m_3} \delta_r \frac{h \chi_{kl}(-1)}{rq^2 \phi(k)}
 \\
 R_k(m_1) R_k(m_2) R_k(m_3) H(\overline{rhk} m_1 m_2 m_3;l) \Big(\frac{m_1 m_2 m_3}{qr} \Big)^{iu}  \Big| du,
\end{multline}
where $h,k, l$ are functions of the other variables, and the star on the sum indicates that \eqref{eq:coprimality} holds.  Lemma \ref{lemma:bilinearGeneral} immediately shows that the most important terms $h=k=1$, $l=q$ give the stated bound, but of course we need to treat all cases.  We may assume that the coefficients are zero unless $m_i \asymp M_i$, and $r \asymp R$.
Next we write $r = hr'$ where $h|q$ and $(r', q/h) = 1$.  
Note that \eqref{eq:coprimality} means $(m_1, h) = (m_2, h) = (m_3, h) = 1$, and $(m_3, r') = 1$.
We also move the sum over $k$ to the outside, giving that the left hand side of \eqref{eq:bilinearWithG} is
\begin{multline}
\label{eq:tacos}
\ll \sum_{hk|q} \frac{h}{R q^2\phi(k)} \int_{|u| \leq U} \Big|\sumstar_{\substack{m_1, m_2, m_3 \\ (m_1 m_2 m_3, q) = k}} 
\medspace \sumstar_{(r', m_3)=1} 
\alpha_{m_1} \beta_{m_2} \gamma_{m_3} \delta_{hr'}
\\
 R_k(m_1) R_k(m_2) R_k(dm_3) H(\overline{r'h^2k} m_1 m_2 m_3; l) \Big(\frac{m_1 m_2 m_3}{r'} \Big)^{iu}  \Big| du,
\end{multline}
where the stars on the sums indicate the already-listed coprimality conditions.  The key point is that these conditions are only of the form $(m_i', s) = 1$ where $s$ is an integer independent of the other $m_i'$, and $r'$ (and similarly for the condition on $r'$).  In this way, we can absorb these conditions into the definition of the coefficients without altering the upper bound on their magnitude. 

The next problem is that the condition $(m_1 m_2 m_3, q) = k$ is in terms of the \emph{product} of the $m_i$ and so these variables are not separated.  However, this is easily solved as follows.  For fixed $k | q$, the condition $(m_1 m_2 m_3, q) = k$ means $(m_1 m_2 m_3, k) = k$ and $(m_1, q/k) = (m_2, q/k) = (m_3, q/k) = 1$.
We can parameterize the solutions to $(m_1 m_2 m_3, k) = k$ as follows.  Suppose $(m_1,k) = k_1$, and write $m_1 = k_1 m_1'$, $k = k_1 k'$, so $(m_1, k') = 1$.  Now $(m_2 m_3, k') = k'$, so we can repeat this argument giving say $(m_2, k') = k_2$, $m_2 = k_2 m_2'$, $k' = k_2 k_3$, and so therefore $(m_3, k_3) = k_3$, so we may write $m_3 = k_3 m_3'$.  In this way we separate the $m_i'$. 

 Therefore \eqref{eq:tacos} is 
\begin{multline}
\ll \sum_{\substack{hk|q \\ k = k_1 k_2 k_3}} \frac{h}{R q^2\phi(k)} \int_{|u| \leq U} \Big|\sumstar_{\substack{m_1', m_2', m_3'}} 
\medspace
\sumstar_{\substack{r' \asymp R/h \\ (r', m_3') =1 }} 
\alpha_{k_1 m_1'} \beta_{k_2 m_2'} \gamma_{ k_3 m_3'} \delta_{hr'}
\\
 R_k(m_1' k_1) R_k(m_2' k_2) R_k(m_3' k_3 ) H(\overline{r'h^2} m_1' m_2' m_3'; l) \Big(\frac{m_1 m_2 m_3}{r'} \Big)^{iu} du \Big|.
\end{multline}
Using $|R_k(m)| \leq (m,k)$, we easily see that 
\begin{equation}
\sumstar_{m_1' \ll M_1/k_1} |\alpha_{k_1 m_1'}|^2  |R_{k_1 k_2 k_3}(k_1 m_1')|^2 \leq k_1^2 \sum_{m_1 \ll M_1} |\alpha_{m_1}|^2,
\end{equation}
and similarly for $m_2'$ and $m_3'$.
Therefore, Lemma \ref{lemma:bilinearGeneral} shows that the left hand side of \eqref{eq:bilinearWithG} is
\begin{multline}
 \ll \sum_{h|q} \sum_{k_1 k_2 k_3 =k |q} \frac{h}{R q^2\phi(k)} \Big(\frac{q}{hk}\Big)^{\frac12+ \varepsilon} \Big(\frac{q U}{hk} + \frac{M_1 M_2}{k_1 k_2} \Big)^{\frac12} \Big(\frac{q U}{hk} + \frac{M_3 R}{k_3 h}\Big)^{\frac12} 
 \\
 \Big( k^2\sum_{d,m_1, m_2, m_3, r} d^{1+\varepsilon} |\alpha_{m_1} \beta_{m_2} \gamma_{dm_3} \delta_{dr}|^2 \Big)^{\frac12}.
\end{multline}
One easily checks that larger values of $h$ and $k$ do not create a larger error term than the case with $h=k=1$, so the proof is complete.
\end{proof}

\section{Completion of the proof}
In this section we combine estimates from Sections \ref{section:separationAnalytic} and \ref{section:largesieve}, and prove Proposition \ref{prop:generalH} which in turn implies Theorem \ref{thm:cubicmoment}.

\begin{proof}[Proof of Proposition \ref{prop:generalH}]
Using \eqref{eq:SGK}, Lemma \ref{lemma:Kbound}, and \eqref{eq:GH}, we obtain
\begin{multline}
 S_{\pm}(N_1, N_2, N_3;C; \BesselTransform_0)
 \ll \sum_{M_1, M_2, M_3 \text{ dyadic}} \int_{|{\bf y}| \leq T^{\varepsilon}} \frac{C^{1/2} \Delta T (N_1 N_2 N_3)^{1/2}}{(M_1 M_2 M_3)^{1/2} V} 
 \\
\int_{|u| \ll U}  \Big| \sum_{m_i \asymp M_i} \sum_{r \asymp C/q} m_1^{iy_1} m_2^{iy_2} m_3^{iy_3} c^{iy_4} G(m_1, m_2, m_3, qr) e_{qr}(-m_1 m_2 m_3) \Big(\frac{m_1 m_2 m_3}{qr}\Big)^{iu} du \Big|,
\end{multline}
plus an error term of size $O(T^{-A})$.  Recall we wish to show the bound \eqref{eq:goalbound}.
We need this for the three choices of $U$ and $V$ appearing in Lemma \ref{lemma:Kbound}, so we do not specialize them yet.  By Lemma \ref{lemma:bilinearWithG} (and its following remark), we bound this by
\begin{equation}
\frac{C^{1/2} \Delta T (N_1 N_2 N_3)^{1/2}}{(M_1 M_2 M_3)^{1/2} V}  \frac{q^{1/2 + \varepsilon} T^{\varepsilon}}{Cq}  \Big(qU + M_1 M_2\Big)^{1/2} \Big(qU + \frac{M_3 C}{q}\Big)^{1/2}\Big(\frac{M_1 M_2 M_3 C}{q} \Big)^{1/2}.
\end{equation}
Without using any specializations yet, this simplifies as
\begin{equation}
\label{eq:SsumSimplifiedBoundNotSpecialized}
 \Delta T (N_1 N_2 N_3)^{1/2} (qT)^{\varepsilon} \Big[ \frac{1}{qV}  \Big(qU + M_1 M_2\Big)^{1/2} \Big(qU + \frac{M_3 C}{q}\Big)^{1/2} \Big].
\end{equation}
As the desired bound $\Delta T (N_1 N_2 N_3)^{1/2} (qT)^{\varepsilon}$ is already visible in front, we treat the part of \eqref{eq:SsumSimplifiedBoundNotSpecialized} inside the square brackets.

First consider Case 1 from Lemma \ref{lemma:Kbound} where $M_i \asymp (N_1 N_2 N_3)^{1/2}/N_i$, $V = T$, and $U$ and $C$ are given by \eqref{eq:Udef} and \eqref{eq:CsizeCos}.  Here we have $qU \ll \frac{q T}{\Delta} T^{\varepsilon}$, $\frac{M_3 C}{q} \ll \frac{N_1 N_2}{\Delta T q} T^{\varepsilon} \ll \frac{qT}{\Delta} T^{\varepsilon} $ (recalling $N_i \ll (qT)^{1+\varepsilon}$), and $M_1 M_2 \asymp N_3 \ll (qT)^{1+\varepsilon}$.  Thus \eqref{eq:SsumSimplifiedBoundNotSpecialized} simplifies as
\begin{equation}
 \Delta^{1/2} T (N_1 N_2 N_3)^{1/2} (qT)^{\varepsilon}.
\end{equation}

Next consider Case 2 from Lemma \ref{lemma:Kbound} with $\frac{M_i N_i}{C} \gg T^{\varepsilon}$.  Here we have $X = \frac{M_1 M_2 M_3}{C}$, $V = T$, $U=X^{1/3} T^{2/3}$, \eqref{eq:CMisizesforSin}, and $M_1 N_1 \asymp M_2 N_2 \asymp M_3 N_3$.  
Using \eqref{eq:CMisizesforSin}, we easily deduce $X \ll T \Delta^{-3 + \varepsilon}$.  
This implies $U \ll \frac{T}{\Delta} T^{\varepsilon}$, exactly as the above Case 1.  Similarly, $\frac{M_3 C}{q} \ll \frac{N_1 N_2}{q T \Delta} T^{\varepsilon} \ll (qT)^{1+\varepsilon}$, as in Case 1, and $M_1 M_2 \ll N_3 \Delta^{-2} T^{\varepsilon}$, which is slightly smaller than in Case 1.  It is now clear that Case 2 gives the same bound as Case 1.

Finally we treat the case of $\frac{M_i N_i}{C} \ll T^{\varepsilon}$ (which holds for all $i$, otherwise $K$ is small), giving that the expression in brackets is, using $X \ll T^{\varepsilon}$
\begin{equation}
 \ll \frac{T^{\varepsilon}}{q} \Big(q + M_1 M_2\Big)^{1/2} \Big(q + \frac{M_3 C}{q}\Big)^{1/2}  \ll \frac{T^{\varepsilon}}{q} \Big(q + \frac{C^2}{N_1 N_2}\Big)^{1/2}\Big(q + \frac{C^2}{N_3 q}\Big)^{1/2} .
\end{equation}
It still remains true that $C \asymp \frac{(N_1 N_2 N_3)^{1/2}}{T}$, so we quickly bound the expression in brackets by $(qT)^{\varepsilon}$, as desired.
\end{proof}


\begin{thebibliography}{99}
 \bibitem[BKY]{BKY} V. Blomer, R. Khan, and M. Young, \emph{Distribution of mass of holomorphic cusp forms.} 
Duke Math. J. 162 (2013), no. 14, 2609--2644. 
\bibitem[Bl]{BlomerGL3} V. Blomer, \emph{Subconvexity for twisted $L$-functions on $GL(3)$}. Amer. J. Math. 134 (2012), no. 5, 1385--1421.

\bibitem[BH]{BH} V. Blomer and G. Harcos, \emph{
Hybrid bounds for twisted $L$-functions.}
J. Reine Angew. Math. 621 (2008), 53--79. 

 \bibitem[Bu]{Burgess} D. A. Burgess, \emph{On character sums and $L$-series, I}, Proc. London Math. Soc. 12 (1962), 193--206.
\bibitem[CI]{CI}  J. B. Conrey and H. Iwaniec, \emph{The cubic moment of central values of automorphic $L$-functions.} Ann. of Math. (2) 151 (2000), no. 3, 1175--1216.
\bibitem[D]{Duke} W. Duke, {\it Hyperbolic distribution problems and half-integral weight Maass forms.}  Invent. Math. 92 (1988), no. 1, 73--90.

\bibitem[Ga]{Gallagher} P. X. Gallagher, {\it 
A large sieve density estimate near $\sigma =1$.}
Invent. Math. 11 1970 329--339. 
\bibitem[GR]{GR} I.S. Gradshteyn, and I.M. Ryzhik, {\it Table of Integrals, Series, and Products}. 
 Translated from the Russian. Sixth edition. Translation edited and with a preface by Alan Jeffrey and Daniel Zwillinger. Academic Press, Inc., San Diego, CA, 2000.
 
\bibitem[Gu]{Guo} J. Guo, \emph{On the positivity of the central critical values of automorphic $L$-functions for $GL(2)$}, Duke Math. J. 83 (1996), 157--190.
 \bibitem[HM]{HarcosMichel} G. Harcos and P. Michel, {\em The subconvexity problem
for Rankin-Selberg $L$-functions and equidistribution of Heegner points.
II.\/} Invent. Math.  \textbf{163}  (2006),  581--655.
 \bibitem[H-B1]{H-B1} D.R. Heath-Brown, {\it Hybrid bounds for Dirichlet L-functions.} Invent. Math. 47 (1978), no. 2, 149--170.
\bibitem[H-B2]{H-B2} D.R. Heath-Brown, {\it Hybrid bounds for Dirichlet L-functions. II.} Quart. J. Math. Oxford Ser. (2) 31 (1980), no. 122, 157--167.
\bibitem[HW]{HuxleyWatt} M. N. Huxley and N. Watt, {\it Hybrid bounds for Dirichlet's L-function.} Math. Proc. Cambridge Philos. Soc. 129 (2000), no. 3, 385--415. 
\bibitem[Iv]{Ivic} A. Ivi{\'c}, \emph{On sums of Hecke series in short intervals.} J. Th\'{e}or. Nombres Bordeaux 13 (2001), no. 2, 453--468.
\bibitem[Iw1]{IwaniecHalf} H. Iwaniec, {\it Fourier coefficients of modular forms of half-integral weight.}
Invent. Math. 87 (1987), no. 2, 385--401. 
\bibitem [Iw2]{IwaniecClassical} H. Iwaniec, \emph{Topics in Classical Automorphic Forms}, Grad. Stud.
      Math., vol 17, Amer. Math. Soc., 1997.

\bibitem[IK]{IK}  H. Iwaniec, E. Kowalski, \emph{Analytic number theory}, AMS Colloquium Publications \textbf{53}, American Mathematical Society 2004.


\bibitem[JM]{JM} M. Jutila and Y. Motohashi, {\it Uniform bound for Hecke $L$-functions.}  Acta Math.  195  (2005), 61--115.

\bibitem[KS]{KatokSarnak} S. Katok and P. Sarnak, {\it Heegner points, cycles and Maass forms.} Israel J. Math. 84 (1993), no. 1-2, 193--227. 


\bibitem[KZ]{KohnenZagier} W. Kohnen and D. Zagier, {\em Values of $L$-series of modular forms at the center of critical strip\/}. Invent. Math. \textbf{64} (1981), 175--198.

\bibitem[Li]{XLi}  Xiaoqing Li, \emph{Bounds for $GL(3)\times GL(2)$ $L$-functions and $GL(3)$ $L$-functions}, Ann. of Math.  \textbf{173} (2011),  301-336.

\bibitem[LMY]{LMY} S.-C. Liu, R. Masri and Mattthew P. Young, {\em Subconvexity and equidistribution of Heegner points in 
the level aspect.\/} Compositio Math. \textbf{149} (2013), 1150--1174.

\bibitem[Lu]{Lu} Q. Lu, \emph{
Bounds for the spectral mean value of central values of $L$-functions.} 
J. Number Theory 132 (2012), no. 5, 1016--1037. 

\bibitem[Pen]{Pe} Z. Peng, \emph{Zeros and central values of automorphic $L$-functions}, Princeton PhD thesis 2001.

\bibitem[Pet]{Petrow} I. Petrow, \emph{A twisted Motohashi formula and Weyl-subconvexity for $L$-functions of weight two cusp forms}, to appear in Mathematische Annalen.

\bibitem[W]{Waldspurger} J.-L. Waldspurger, {\it Sur les coefficients de Fourier des formes modulaires de poids demi-entier.} J. Math. Pures Appl. (9) 60 (1981), no. 4, 375--484. 

\bibitem[Y]{YoungQUE} M. Young, \emph{The quantum unique ergodicity conjecture for thin sets},  Adv. Math. 286 (2016), 958--1016.

\bibitem[Z]{Zhang} S. W. Zhang, {\em Gross--Zagier formula for $GL_2$\/}. Asian J. Math. \textbf{5} (2001), 183--290.
\end{thebibliography}
\end{document}